\documentclass[11pt,reqno]{amsart}

\usepackage[section]{multi-dim}

\title[Persistent anisotropy in multi-dimensional branching Brownian motion]{Persistent anisotropy in multi-dimensional\\ branching Brownian motion}
\author{Cole Graham, Leonid Mytnik, Lenya Ryzhik}

\contact{CG}{Department of Mathematics, University of Wisconsin--Madison, Madison, WI, USA}{graham@math.wisc.edu} %
\contact{LM}{Faculty of Data and Decision Sciences, Technion, Haifa, Israel}{leonidm@technion.ac.il} %
\contact{LR}{Department of Mathematics, Stanford University, Stanford, CA, USA}{ryzhik@stanford.edu}

\date{\today}

\begin{document}
\begin{abstract}
  We study angular variation in the frontier of branching Brownian motion (BBM) in multiple dimensions.
  The time average of this frontier is tied to the long-time limit of the critical derivative martingale, which is a random integrable function on the sphere.
  We show that this function is nowhere locally bounded.
  As a consequence, BBM exhibits strong persistent anisotropy: there are dense arbitrarily large gaps between the BBM frontier in different directions.
\end{abstract}

\maketitle

\section{Introduction}

Branching Brownian motion (BBM) is a spatial branching process in which particles randomly move and reproduce.
As more particles are born, the population expands outward along a moving frontier.
The statistics of BBM near this frontier have a universal character common to a large class of ``log-correlated fields''~\cite{Arguin17}.
Consequently, the position and nature of the BBM frontier have attracted a great deal of attention in one dimension~\cite{Bramson78,Bramson83,LalSel87,ArgBovKis13a,AidBerBruShi13}.
Recently, there has been a surge of interest in the extremal particles of branching Brownian motion in multiple spatial dimensions~\cite{Mallein15,StaBerMal21,KimLubZei23,BerKimLubMalZei24,KimZei25}.
In this spirit, we investigate a fundamentally multi-dimensional phenomenon: the angular variation in the BBM frontier.
We show that when~$d \geq 2$, BBM advances considerably farther in some directions than in others.

We study branching Brownian motion informally defined through the following inductive formula.
An initial particle performs a Brownian motion from the origin in $\R^d$ with diffusivity matrix $\sqrt{2} \op{Id}$.
At an independent standard exponential time, the particle splits into two children, each of which independently performs the above process from the parent's position.
We thus obtain a randomly growing population~$\m{P}_t$ of particles at a given time $t$.
For each particle $p \in \m{P}_t$, let $X_t(p) \in \R^d$ denote its position.
For rigorous constructions of BBM, see \cite{IkeNagWat68,HarHar09,Chauvin91}.

Fix a direction $\theta \in S^{d-1}$.
Then $\theta \cdot X$ is the projection of the branching Brownian motion in $\R^d$ onto direction $\theta$.
It agrees in law with a one-dimensional BBM.
To study extremal behavior, consider the maximum
\begin{equation*}
  M_t(\theta) \coloneqq \max\{\theta \cdot X_t(p) : p \in \m{P}_t\}.
\end{equation*}
Geometrically, $M_t(\theta)$ is the displacement of the supporting hyperplane in direction~$\theta$ of the point set $\{X_t(p)\}_{p \in \m{P}_t}$.

Bramson's seminal 1D work~\cite{Bramson78} shows that for each fixed $\theta \in S^{d-1}$,
\begin{equation}
  \label{eq:Bramson}
  M_t(\theta) = 2t - \frac{3}{2} \log t + \m{O}_\P(1),
\end{equation}
where $\m{O}_\P(1)$ denotes a tight random variable depending on $t$ and $\theta$.
This tight remainder can be expressed as the sum of a time-independent random shift and an asymptotically stationary mean-zero fluctuation.
Lalley and Sellke~\cite{LalSel87} first showed the essence of this decomposition, and asymptotic stationarity can be traced to the convergence of the extremal point process due independently to Arguin, Bovier, and Kistler~\cite{ArgBovKis13a} and A\"id\'ekon, Berestycki, Brunet, and Shi~\cite{AidBerBruShi13}.

We are interested in the constant-in-time random shift, which determines the persistent ``bias'' of the BBM in direction $\theta$.
One can isolate this shift by averaging over time to suppress the stationary fluctuation.
Flath recently established the following ergodic theorem~\cite{Flath25}: almost surely,
\begin{equation}
  \label{eq:ergodic}
  \frac{1}{T} \int_0^T \big[M_t(\theta) - 2 t + \tfrac{3}{2} \log t\big] \d t \to \log Z_\infty(\theta) + \f{c} \quad \text{as } T \to \infty.
\end{equation}
Here $\f{c}$ is a deterministic constant and $Z_\infty(\theta)$ is the almost-surely positive limit of the critical derivative martingale in direction $\theta$:
\begin{equation}
  \label{eq:deriv}
  Z_t(\theta) \coloneqq \sum_{p \in \m{P}_t} [2t - \theta \cdot X_t(p)] \e^{-[2t - \theta \cdot X_t(p)]}.
\end{equation}
We note that \eqref{eq:ergodic} is an integrated form of an ergodic theorem conjectured in~\cite{LalSel87} and established by Arguin, Bovier, and Kistler~\cite{ArgBovKis13b}.
These works concern 1D BBM, so \eqref{eq:ergodic} holds for each fixed $\theta$ due to equality in law with the projected process~$\theta \cdot X$.
It then follows from Fubini's theorem that with probability one, there exists a random full-measure subset of~$S^{d-1}$ within which~\eqref{eq:ergodic} holds for all $\theta$.
We can likewise ensure that $Z_t(\theta)$ converges to $Z_\infty(\theta) > 0$ in the same subset.

Due to \eqref{eq:ergodic}, the derivative martingale limit $Z_\infty$ determines the persistent bias of the BBM frontier in direction $\theta$.
Thus, the size and regularity of $Z_\infty$ governs the relative advance of the branching process in different directions.
Stasi\'{n}ski, Berestycki, and Mallein have shown that $Z_\infty \in L^1(S^{d-1})$ almost surely~\cite{StaBerMal21}.
Our main result shows that $Z_\infty$ is not substantially more regular.
\begin{theorem}
  \label{thm:main}
  If $d \geq 3$, then almost surely, for every nonempty open $U \subset S^{d-1}$, $\norm{Z_\infty}_{L^p(U)} = \infty$ for all $p \geq \frac{d-1}{d-2}$.
  If $d = 2$, the same holds for $p = \infty$.
\end{theorem}
\noindent
The sharpness of the threshold $p = \frac{d - 1}{d - 2}$ remains an interesting open question.

Taking $p = \infty$, Theorem~\ref{thm:main} implies that in multiple dimensions, $Z_\infty$ is nowhere locally bounded.
Thus in any open set of directions, BBM advances arbitrarily farther in one direction relative to another:
\begin{corollary}
  \label{cor:main}
  Let $d \geq 2$.
  Almost surely, for each $L,\eps > 0$ and nonempty open $U \subset S^{d-1}$, there exist $\theta_1,\theta_2 \in U$ such that
  the long-term fraction of time during which $M_t(\theta_1) \geq M_t(\theta_2) + L$ is at least $1 - \eps$.
  That is,
  \begin{equation}
    \label{eq:liminf}
    \liminf_{T \to \infty} \frac{1}{T} \int_0^T \tbf{1}\bigl(M_t(\theta_1) \geq M_t(\theta_2) + L\bigr) \d t \geq 1 - \eps.
  \end{equation}
\end{corollary}
\noindent
In this sense, the BBM frontier is quite irregular.

Our approach builds on the following intuition: if a particle makes an unusual radial excursion, its descendants impart a permanent boost to $Z_t$ in nearby directions.
To prove Theorem~\ref{thm:main}, we show that the BBM makes suitably large excursions in a dense set of directions.
We take inspiration from the work of Maillard and Pain~\cite{MaiPai19}, who used such excursions in one dimension to characterize the $\m{O}(t^{-1/2})$ fluctuations of $Z_t$.
In higher dimensions, there are more directions and thus more excursions, so this effect contributes at leading rather than vanishing order.

We are interested in particles that travel well beyond the typical frontier.
By~\eqref{eq:Bramson}, in any fixed direction, the BBM frontier is typically at distance $2t - \frac{3}{2} \log t$ from the origin.
However, Mallein~\cite{Mallein15} showed that there are random, time-varying directions in which particles are considerably farther, at distance $2t + \frac{d-4}{2} \log t$.
If we wait for an exceptional \emph{time}, particles will travel farther still.
Indeed in one dimension, Hu and Shi~\cite{HuShi09} and A\"id\'ekon and Shi~\cite{AidShi14} showed that particles occasionally travel $\log t$ beyond the typical frontier.

This suggests that in higher dimensions, we should occasionally find particles at the radius
\begin{equation}
  \label{eq:numerology}
  2t + \frac{d-4}{2} \log t + \log t = 2t + \frac{d-2}{2} \log t.
\end{equation}
As a start, it is relatively straightforward to show that particles occasionally reach $2t + \al \log t$ for any $\al < \frac{d-2}{2}$.
Their contributions to $Z_\infty$ are sufficiently large to prove Theorem~\ref{thm:main} when $d \geq 3$ and $p > \frac{d-1}{d-2}$.
However, the critical dimension $2$ (or the threshold $p = \frac{d-1}{d-2}$) requires much finer control.
We show that particles do reach the radius in \eqref{eq:numerology}, and in fact go slightly farther.
This rests on a Kolmogorov-type integral test for the maximal radial displacement of multidimensional branching Brownian motion.
To state it, we introduce the notation $\big(R_t(p),\Theta_t(p)\big)$ for the radial and angular coordinates of the position $X_t(p)$ of particle $p$ at time $t$.
\begin{theorem}
  \label{thm:LIL}
  Let $\beta \in \m{C}^2\big([0, \infty); \R\big)$ satisfy $t \dot{\beta}_t \to 0$ as $t \to \infty$.
  If 
  \begin{equation}
    \label{eq:integral-condition}
    \int_1^\infty \frac{\dn r}{r \exp \beta_r} = \infty,
  \end{equation}
  then almost surely, for every nonempty open $U \subset S^{d-1}$, there exist sequences $t_n \nearrow \infty$ and $p_n \in \m{P}_{t_n}$ such that for all $n \in \N$,
  \begin{equation}
    \label{eq:large-excursion}
    R_{t_n}(p_n) = 2t_n + \frac{d-2}{2} \log t_n + \beta_{t_n} \And \Theta_{t_n}(p_n) \in U.
  \end{equation}
  If \eqref{eq:integral-condition} does not hold, then almost surely there exists $\tau > 0$ such that for all $t > \tau,$
  \begin{equation*}
    \max_{p \in \m{P}_t} R_t(p) < 2t + \frac{d-2}{2} \log t + \beta_t.
  \end{equation*}
\end{theorem}
\noindent
The criterion \eqref{eq:integral-condition} was identified by Hu~\cite{Hu15} for 1D branching random walks, and we extend it here to BBM in $\R^d$.
The choice $\beta_t = \log \log t$ satisfies \eqref{eq:integral-condition} and allows us to prove the borderline cases of Theorem~\ref{thm:main}.
We once again emphasize that the fine result in Theorem~\ref{thm:LIL} is only necessary to treat $d = 2$ or $p = \frac{d-1}{d-2}$.
When~$d \geq 3$ and $p > \frac{d-1}{d-2}$, much shorter arguments suffice, as discussed after \eqref{eq:informal} and in Remark~\ref{rem:3D} below.
\medskip

Previous works on exceptional excursions~\cite{HuShi09,AidShi14,Hu15} have largely relied on spine decompositions~\cite{Lyons94} and the second moment method.
We take a different approach, by expressing probabilities of interest in terms of solutions to the Fisher--KPP equation, a deterministic parabolic PDE.
This strategy is well known in one dimension, and led to Bramson's result in \eqref{eq:Bramson} \cite{Bramson78}.
However, the connection between branching processes and PDEs has received relatively little attention in multiple dimensions.
To our knowledge, this is the first work that uses intrinsically multi-dimensional Fisher--KPP solutions to study branching Brownian motion.
Our proof of Theorem~\ref{thm:LIL} combines such solutions with results on the derivative martingale from~\cite{BerKimLubMalZei24}.
We thus take a hybrid approach, drawing on both PDE constructions and purely probabilistic estimates to produce large excursions.

The duality between BBM and the Fisher--KPP equation involves a time reversal, so the long-time behavior of BBM is tied to ancient solutions of Fisher--KPP.
In order to treat arbitrary open sets of directions as in Theorem~\ref{thm:main}, these solutions are necessarily anisotropic and do not have a direct one-dimensional analog.
For example, our proof of the localization of mass in $Z_t$ (Proposition~\ref{prop:local} below) rests on a detailed analysis of the distinct behavior of a Fisher--KPP solution in one direction relative to others.
This line of inquiry bears some resemblance to PDE literature such as constructions of conical traveling waves in~\cite{HamMon00,HamMonRoq05} and the partial classification of entire solutions to Fisher--KPP in~\cite{HamNad01}, but the connection is indirect.

\subsection*{Organization}

In the next section, we outline our approach.
In Section~\ref{sec:excursions}, we show that BBM makes large excursions assuming a crucial PDE estimate.
We prove this estimate over the next two sections, which develop 1D (Section~\ref{sec:1D-excursions}) and multi-D (Section~\ref{sec:excursions-proof}) estimates for the relevant Fisher--KPP solution.
In Section~\ref{sec:necessary}, we prove a matching upper bound, showing that BBM does not travel much farther.
This completes the proof of Theorem~\ref{thm:LIL}.

Turning to Theorem~\ref{thm:main}, we show that large excursions produce significant mass in $Z_\infty$ (Section~\ref{sec:total}) that is sharply localized (Section~\ref{sec:local}).
This implies Theorem~\ref{thm:main}.
In an appendix, we observe that results of~\cite{BerKimLubMalZei24} are robust to modest variations.

\subsection*{Acknowledgments}

CG was supported by the National Science Foundation (NSF) through DMS-2516786.
LR was supported by NSF grants DMS-2205497 and DMS-2510166. 
LM was supported in part by the ISF grant No. 1985/22.

\section{Proof outline}
\label{sec:outline}

In this section, we explain the structure of our argument and prove Theorem~\ref{thm:main} and Corollary~\ref{cor:main}, assuming certain auxiliary results that we defer to later sections.

\subsection{Excursions and the derivative martingale}

As noted above, our proof rests on the observation that large radial excursions produce large values in $Z_\infty$ with high probability.
We now offer an informal explanation of why this is the case.

Define a radial threshold
\begin{equation}
  \label{eq:threshold}
  b_t \coloneqq 2t + \frac{d-2}{2} \log t + \beta_t
\end{equation}
fine-tuned by a slowly varying function $t \mapsto \beta_t$ satisfying the hypotheses of Theorem~\ref{thm:LIL} and \eqref{eq:integral-condition}.
We will ultimately take $\beta_t = \log \log t$, but for the moment we allow it to be an arbitrary element of this class.

Given a nonempty open set of directions $U \subset S^{d-1}$, Theorem~\ref{thm:LIL} ensures that particles cross the radial threshold $b_t$ infinitely often through the sector $U$.
In particular, if we fix a large time $T \gg 1$, then almost surely there exists a random earliest stopping time $\h t \geq T$ and a particle $\h p \in \m{P}_{\h t}$ making a large excursion in $U$:
\begin{equation*}
  R_{\h t}(\h p) = b_{\h t} = 2\h t + \frac{d-2}{2} \log \h t + \beta_{\h t} \And \Theta_{\h t}(\h p) \in U.
\end{equation*}
Both $\h t$ and $\h p$ depend on $U$ and $T$, but we suppress this dependence for legibility.

Now recall the derivative martingale $Z_t(\theta)$ defined in \eqref{eq:deriv}.
If we omit the linear prefactor in \eqref{eq:deriv}, we obtain the critical ``additive'' martingale $W_t$, which vanishes in the long-time limit.
Under the so-called Seneta--Heyde scaling, $\sqrt{t} W_t \to C Z_\infty$ in probability as $t \to \infty$~\cite{HuShi09,AidShi14}.
This suggests that the derivative martingale $Z_t$ in \eqref{eq:deriv} is dominated by particles roughly $\sqrt{t}$ behind the BBM frontier.
This provides considerable leeway in analyzing $Z_t$: we can drop particles outside this window (as in $Y^{\text{win}}$ in \eqref{eq:Y-window}) and modify the prefactor by $o(\sqrt{t})$ (as in $Y$ and $\h Y$ below) without changing the long-time limit.
We make frequent use of this flexibility, routinely modifying $Z_t$ to produce convenient properties while preserving its essential nature.
Such modifications are not true martingales, but this plays no role in our analysis.

As a first example, we work with a variant of $Z_t$ similar to one introduced in~\cite{BerKimLubMalZei24}.
Define the random measure on $S^{d-1}$
\begin{equation}
  \label{eq:shaved}
  Y_t(\dn\theta)\coloneqq \pi^{\frac{d-1}{4}} \sum_{p \in \m{P}_t} ([b_t - R_t(p)]_+ + 1) \e^{-[2t + \frac{d-1}{2} \log t - R_t(p)]} \delta_{\Theta_t(p)}(\dn\theta).
\end{equation}
The intuition is as follows: if we integrate $Z_t(\theta)$ from \eqref{eq:deriv} against a smooth function of $\theta$, we can use Laplace's method to identify the asymptotic contribution of a given particle to the integral.
This produces a factor of $(\sqrt{\pi}/t)^{(d-1)/2}$, which appears in the point measure \eqref{eq:shaved}.
Thus \eqref{eq:deriv} and \eqref{eq:shaved} behave similarly when integrated against test functions.
In addition, we have modified the linear prefactor in a manner that preserves the leading contribution of particles roughly $\sqrt{t}$ behind the frontier.
As discussed above, this change has a negligible effect at long times:
\begin{theoremalph}[\cite{BerKimLubMalZei24}]
  \label{thm:convergence}
  As $t \to \infty$, the measure $Y_t$  converges weakly in probability to the limit $Z_\infty$ of the projected derivative martingale in \eqref{eq:deriv}.
\end{theoremalph}
\noindent
This is a minor variation on Theorem~1.4 in~\cite{BerKimLubMalZei24}; we include a proof in Appendix~\ref{sec:convergence}.

We wish to show that $Y_t$ is consistently large in certain directions.
Because its summands are nonnegative, it suffices to consider only those $p$ that descend from the particle~$\h p$ defined above.
We let $\h{\m{P}}_t$ denote this subpopulation and introduce another random measure
\begin{equation}
  \label{eq:restricted}
  \h{Y}_t(\dn\theta) \coloneqq \pi^{\frac{d-1}{4}} \sum_{p \in \h{\m{P}}_t} ([\h{b}_t - R_t(p)]_+ + 1) \e^{-[\h{b}_t - R_t(p)]} \delta_{\Theta_t(p)}(\dn\theta) \quad \text{for } t \geq \h t,
\end{equation}
where
\begin{equation}
  \label{eq:b-hat}
  \h{b}_t \coloneqq 2t + \frac{d-1}{2} \log t - \frac{1}{2} \log \h t + \beta_{\h t}.
\end{equation}
Note that the curves $b$ and $\h b$ agree at time $\h t$.

The exponent in \eqref{eq:restricted} differs from that in \eqref{eq:shaved} by $- \frac{1}{2} \log \h t + \beta_{\h t}$, which is chosen so that $\h Y_t$ has mass of order $1$ at its initial time $\h t$.
The difference in the linear prefactor in \eqref{eq:restricted} plays a negligible role (see Appendix~\ref{sec:convergence}), so
\begin{equation}
  \label{eq:positive}
  Y_t \geq \h t^{-1/2} \e^{\beta_{\h t}} \h Y_t + \smallO_{t \to \infty}(1).
\end{equation}
At time $\h t$, the set $\h{\m{P}}_{\hat t}$ consists of the single particle $\h p$, and $R_{\hat t}(\h p)=b_{\hat{t}}= \h b_{\h t}$.
Hence $\h Y_{\h t}(S^{d-1}) = \pi^{(d-1)/4}$.
By design, both $Y$ and $\h Y$ are approximately martingales.
We thus expect the total mass of $\h Y$ to remain bounded away from zero.
To express this carefully, let $\m{G}_{U,T}$ denote the $\sigma$-algebra containing all information about particles until they first cross the threshold $b_t$ through $U$ after time $T$; for a rigorous definition, see~\cite{Chauvin91}.
Conditioning on $\m{G}_{U,T}$ allows us to discuss probabilities at times $t \geq \h t$.
\begin{lemma}
  \label{lem:total}
  For all $\eps > 0$, there exists $m > 0$ such that for all $t \geq \h t$,
  \begin{equation*}
    \P\big[\h Y_t(S^{d-1}) \geq m \mid \m{G}_{U,T}\big] \geq 1 - \eps.
  \end{equation*}
\end{lemma}
\noindent
In Section~\ref{sec:total}, we express this statement in terms of the Laplace transform of $\h Y_t$, which we analyze using a radially-symmetric solution of the Fisher--KPP equation.

Due to \eqref{eq:positive}, this bound implies that with high probability $Y_t(S^{d-1}) \gtrsim \h t^{-1/2} \e^{\beta_{\h t}}$, which does \emph{not} suffice to show that $\norm{Z_\infty}_{L^\infty} = \infty$.
Indeed, no argument based on total mass alone can work, for~\cite{StaBerMal21} showed that $\norm{Z_\infty}_{L^1} < \infty$ a.s.
Thus, to show that the $L^\infty$ norm diverges, we must argue that some of the mass of $Z_\infty$ remains localized in a small region.
This is plausible, because $\h Y_t$ is dominated by particles near the BBM frontier at time $t$.
In order to reach the frontier from the location of their parent $\h p$, they cannot afford to move much in the angular coordinate.
In fact, $\h{Y}_t$ remains localized to scale $\h t^{-1/2}$ around the angle $\h\theta \coloneqq \h \Theta_{\h t}(\h p) \in U$.
\begin{proposition}
  \label{prop:local}
  Let $\Gamma$ denote a ball in $S^{d-1}$ centered at $\h \theta$.
  For all $K > 0$, there exist $A(K),C(d,K) > 0$ such that if $\Gamma$ has radius $A \h t^{-1/2} \log^{1/2} \h t$, then for all $T \geq C(d,K)$ and~$t \geq \h t$,
  \begin{equation*}
    \P\big[\h Y_t(\Gamma^\cc) \geq \h t^{-K} \mid \m{G}_{U,T}\big] \leq  T^{-K}.
  \end{equation*}
\end{proposition}
\noindent
In Section~\ref{sec:local}, we show this localization using a Fisher--KPP solution resembling a radially inbound traveling wave that is missing mass in the angular sector $\Gamma$.

\medskip
We now have the pieces in place to show that $Z_\infty$ has large values in $U$.
We use the notation $f \gtrsim g$ to indicate that $f \geq cg$ for some constant $c > 0$, and $f \asymp g$ indicates that $f \gtrsim g$ and $g \gtrsim f$.

Combining Lemma~\ref{lem:total} and Proposition~\ref{prop:local}, we see that $\h Y_t(\Gamma) \gtrsim 1$ in a random neighborhood $\Gamma$ of volume
\begin{equation}
  \label{eq:volume}
  \op{vol}_{S^{d-1}}(\Gamma) \asymp (\log \h t)^{\frac{d-1}{2}} \h t^{-\frac{d-1}{2}}.
\end{equation}
By \eqref{eq:positive},
\begin{equation}
  \label{eq:mass}
  Y_t(\Gamma) \gtrsim \h t^{-1/2} \e^{\beta_{\h t}}.
\end{equation}
Combining \eqref{eq:volume} and \eqref{eq:mass}, we see that the \emph{mean} of $Y_t$ on $\Gamma$ satisfies
\begin{equation*}
  \op{vol}_{S^{d-1}}(\Gamma)^{-1}  Y_t(\Gamma) \gtrsim \e^{\beta_{\h t}} (\log \h t)^{-\frac{d-1}{2}} \h t^{\frac{d-2}{2}}.
\end{equation*}
Sending $t \to \infty$ and using Theorem~\ref{thm:convergence}, the same holds for $Z_\infty$.
The mean is a lower bound on the essential supremum and $\Gamma\subset U$, so
\begin{equation}
  \label{eq:informalbis}
  \norm{Z_\infty}_{L^\infty(U)} \gtrsim \e^{\beta_{\h t}} (\log \h t)^{-\frac{d-1}{2}} \h t^{\frac{d-2}{2}}.
\end{equation}
If $d \geq 3$, it suffices to take 
\begin{equation}
  \label{26jun402}
  \beta_t = - \log \log t,
\end{equation}
and in $d=2$ we can take 
\begin{equation}
  \label{26jun404}
  \beta_t = \log \log t,
\end{equation}
to ensure that in both cases, the right side of \eqref{eq:informalbis} is increasing in $\h t$.
Since $\h t$ is the earliest excursion time after $T$, we have $\h t \geq T$.
This yields
\begin{equation}
  \label{eq:informal}
  \norm{Z_\infty}_{L^\infty(U)} \gtrsim \e^{\beta_T} (\log T)^{-\frac{d-1}{2}} T^{\frac{d-2}{2}}.
\end{equation}
With the above choices of $\beta_t$, we can send $T\to +\infty$
to conclude that
\begin{equation*}
  \norm{Z_\infty}_{L^\infty(U)} = \infty.
\end{equation*}
The difference in sign between the extra drifts in $d\ge 3$ \eqref{26jun402} and $d = 2$ \eqref{26jun404} means that we require the full strength of Theorem~\ref{thm:LIL} in $d=2$.
This greatly complicates the paper, as it is much more difficult to produce the excursions in \eqref{eq:large-excursion} when $\beta_t = +\log \log t$ rather than $-\log \log t$.
We explain this added difficulty in Remark~\ref{rem:3D} below.

\subsection{Proofs of main results}
We now expand this explanation into a rigorous proof of Theorem~\ref{thm:main} given Theorem~\ref{thm:LIL}.
\begin{proof}[Proof of Theorem~\ref{thm:main}]
  Let $d \geq 2$ and 
  \begin{equation}
    \label{26apr910}
    p = 
    \begin{cases}
      \frac{d-1}{d-2} & \text{if } d \geq 3,\\
      \infty & \text{if } d = 2.
    \end{cases}
  \end{equation}
  Let $\beta$ be a smooth function such that
  \begin{equation}
    \label{26apr906}
    \beta_t = \log \log t
  \end{equation}
  for all $t \geq \e$.
  We use this choice of $\beta$ in \eqref{eq:threshold}, and it satisfies the hypotheses of Theorem~\ref{thm:LIL} and \eqref{eq:integral-condition}.
  Fix $\eps \in (0, 1)$ and let $U \subset S^{d-1}$ and $U' \Subset U$ be nonempty, open, deterministic sets.
  (We treat random sets at the end of the proof.)
  Take also some
  \begin{equation}
    \label{26apr902}
    T \geq \frac{2}{\eps}
  \end{equation}
  such that
  \begin{equation}
    \label{eq:dist}
    \op{dist}_{S^{d-1}}(U',U^\cc) > A\sqrt{\frac{\log T}{T}},
  \end{equation}
  with  $A = A(1)$ given by Proposition~\ref{prop:local}.
  By Theorem~\ref{thm:LIL}, particles almost surely cross the radial threshold $b_t$ through the angular sector $U'$ after time $T$.
  Let $\h t \geq T$ denote the earliest such time, with crossing particle $\h p \in \m{P}_{\h t}$ and angle $\h \theta \in U'$.
  This defines $\h Y$ by \eqref{eq:restricted}.
  Let $\Gamma \subset S^{d-1}$ be the open ball of radius $A\h t^{-1/2} \log^{1/2} \h t$ around~$\h \theta$ in the standard metric on $S^{d-1}$.
  By \eqref{eq:dist}, $\Gamma \Subset U$.

  Now define
  \begin{equation}
    \label{26apr904}
    \mu \coloneqq (\log \h t)^{-1/4}.
  \end{equation}
  If $\h Y_t(S^{d-1}) \geq 2 \mu$, then either $\h Y_t(\Gamma) \geq \mu$ or $\h Y_t(\Gamma^\cc) \geq \mu$.
  Taking a union bound and rearranging, we see that
  \begin{equation}
    \label{eq:diff}
    \P[\h Y_t(\Gamma) \geq \mu \mid \m{G}_{U',T}] \geq \P[\h Y_t(S^{d-1}) \geq 2 \mu \mid \m{G}_{U',T}] - \P[\h Y_t(\Gamma^\cc) \geq \mu \mid \m{G}_{U',T}]
  \end{equation}
  for all $t \geq \h t$.
  By Lemma~\ref{lem:total}, there exists $m > 0$ such that for all $t \geq \h t$, $\h Y_t(S^{d-1}) \geq m$ with conditional probability at least $1 - \eps/2$.
  Possibly increasing $T$ so that
  \begin{equation*}
    \mu \leq (\log T)^{-1/4} \leq \frac{m}{2},
  \end{equation*}
  the first term in the right side of \eqref{eq:diff} is at least $1 - \eps/2$.
  Because $\mu \geq \h t^{-1}$, Proposition~\ref{prop:local} and \eqref{26apr902} 
  ensure that the second term in \eqref{eq:diff} is at most $T^{-1} \leq \eps/2$.
  We thus obtain
  \begin{equation*}
    \P[\h Y_t(\Gamma) \geq \mu \mid \m{G}_{U',T}] \geq 1 - \eps \ForAll t \geq \h t.
  \end{equation*}
  Recall that $\h Y$ is related to $Y$ through \eqref{eq:positive}, and $Y_t$ converges weakly in probability to $Z_\infty$ as $t \to \infty$ (Theorem~\ref{thm:convergence}).
  Weak convergence of finite measures on $S^{d-1}$ is metrizable, so we can extract a subsequence of $(Y_t)_{t \geq 0}$ that converges weakly to $Z_\infty$ almost surely.
  Sending $t \to \infty$ along this subsequence, \eqref{eq:positive}, \eqref{26apr906}, and \eqref{26apr904} imply that
  \begin{equation}
    \label{eq:L1}
    \P\big[Z_\infty(\bar\Gamma) \geq \h t^{-1/2} \log^{3/4} \h t\big] \geq 1 - \eps.
  \end{equation}
  By H\"older's inequality, we have
  \begin{equation}
    \label{26apr914}
    \norm{Z_\infty}_{L^p(U)} \geq \op{vol}_{S^{d-1}}(\Gamma)^{-(1 - 1/p)} Z_\infty(\bar\Gamma).
  \end{equation}
  The Euclidean approximation of the spherical metric implies that
  \begin{equation*}
    \op{vol}_{S^{d-1}}(\Gamma) \asymp \left(\frac{\log \h{t}}{\h t}\right)^{\frac{d-1}{2}}.
  \end{equation*}
  The choice of $p$ in \eqref{26apr910} yields $1 - \frac{1}p = \frac{1}{d-1}$, so
  \begin{equation}
    \label{eq:metric}
    \op{vol}_{S^{d-1}}(\Gamma)^{-(1 - 1/p)} \geq c \h t^{1/2} \log^{-1/2} \h t
  \end{equation}
  for some constant $c(d) > 0$.
  Thus if $Z_\infty(\bar\Gamma) \geq \h t^{-1/2} \log^{3/4} \h t$, then we can combine \eqref{26apr914}, \eqref{eq:metric}, and $\h t \geq T$ to obtain
  \begin{equation*}
    \norm{Z_\infty}_{L^p(U)} \geq c \h t^{1/2} \log^{-1/2} \h t \cdot \h t^{-1/2} \log^{3/4} \h t \geq c \log^{1/4} \h t \geq c \log^{1/4} T.
  \end{equation*}
  That is,
  \begin{equation*}
    Z_\infty(\bar\Gamma) \geq \h t^{-1/2} \log^{3/4} \h t \enspace\implies\enspace \norm{Z_\infty}_{L^p(U)} \geq c \log^{1/4} T.
  \end{equation*}
  Using \eqref{eq:L1}, we see that
  \begin{equation*}
    \P\big(\norm{Z_\infty}_{L^p(U)} \geq c \log^{1/4}T\big) \geq 1 - \eps.
  \end{equation*}  
  Sending $T \to \infty$,
  \begin{equation*}
    \P\big(\norm{Z_\infty}_{L^p(U)} = \infty\big) \geq 1 - \eps.
  \end{equation*}
  Because $\eps \in (0, 1)$ was arbitrary, $\norm{Z_\infty}_{L^p(U)} = \infty$ almost surely.
  H\"older's inequality implies that $\norm{Z_\infty}_{L^{p'}(U)} = \infty$ almost surely for all $p' \geq p$ as well.
  Now suppose the set $U$ is random.
  Let $(U_i)_{i \in \N}$ be a deterministic countable base of $S^{d-1}$.
  Almost surely, $\norm{Z_\infty}_{L^{p'}(U_i)} = \infty$ for all $i$.
  Since any (even random) open set contains some $U_i$, the theorem follows.
\end{proof}
\begin{proof}[Proof of Corollary~\ref{cor:main}]
  We rely on the one-dimensional ergodic theorem of Arguin, Bovier, and Kistler~\cite{ArgBovKis13b}, which, in the multi-dimensional context, states that for each $\theta \in S^{d-1}$ and $x \in \R$, almost surely
  \begin{equation}
    \label{eq:ergodic-occupation}
    \lim_{T \to \infty} \frac{1}{T} \int_0^T \tbf{1}\big(M_s(\theta) - 2s + \tfrac 3 2 \log s \leq x\big) \d s = \exp\left(-Z_\infty(\theta) \e^{\f{c}-x}\right),
  \end{equation}
  where $\f{c} \in \R$ is the constant in \eqref{eq:ergodic}.
  Applying this result to a countable dense subset of $x \in \R$ and using monotonicity and continuity, it follows that almost surely, \eqref{eq:ergodic-occupation} holds for all $x \in \R$.
  This refinement will allow us to apply \eqref{eq:ergodic-occupation} to random $x$.
  Drawing on Fubini, we conclude that there exists $A\subset \Omega$ with $\P(A)=1$ such that for every $\omega\in A$, there is a full-measure subset $E_\omega \subset S^{d-1}$ such that for all $\theta \in E_\omega$, $0<Z_\infty(\theta)<\infty$ and \eqref{eq:ergodic-occupation} holds for all $x \in \R$.

  Let $B\subset \Omega$ be such that $\P(B)=1$ and on $B$ the conclusions of Theorem 1.1 hold.
  Fix $\omega\in A\cap B$ and choose arbitrary $L,\eps > 0$ and a nonempty open set $U \subset S^{d-1}$.
  Let $K(\eps) > 0$ satisfy
  \begin{equation}
    \label{eq:tails}
    \exp(-\e^{\f{c} + K}) \leq \frac{\eps}{2} \And \exp(-\e^{\f{c} - K}) \geq 1 - \frac{\eps}{2}.
  \end{equation}
  By Theorem~\ref{thm:main}, there exist $\theta_1,\theta_2 \in U \cap E_\omega$ such that
  \begin{equation}
    \label{eq:large}
    Z_\infty(\theta_1) \geq \e^{L + 2K} Z_\infty(\theta_2) > 0.
  \end{equation}
  To simplify notation, define
  \begin{equation*}
    \h{M}_t(\theta) \coloneqq M_t(\theta) - 2t + \frac 3 2 \log t - \log Z_\infty(\theta).
  \end{equation*}
  Then if $\h{M}_t(\theta_1) \geq -K$ and $\h{M}_t(\theta_2) \leq K$, \eqref{eq:large} yields
  \begin{equation*}
    M_t(\theta_1) - M_t(\theta_2) = \h{M}_t(\theta_1) - \h{M}_t(\theta_2) + \log\frac{Z_\infty(\theta_1)}{Z_\infty(\theta_2)} \geq -2K + L + 2K = L.
  \end{equation*}
  Taking the complement and a union bound, we see that
  \begin{equation}
    \label{eq:union}
    \tbf{1}(M_t(\theta_1) - M_t(\theta_2) \geq L) \geq 1 - \tbf{1}(\h{M}_t(\theta_1) < -K) - \tbf{1}(\h{M}_t(\theta_2) > K).
  \end{equation}
  We now apply \eqref{eq:ergodic-occupation} with $x = \log Z_\infty(\theta_1) - K$ and $x = \log Z_\infty(\theta_2) + K$ to conclude from \eqref{eq:tails} that
  \begin{equation*}
    \lim_{T \to \infty} \frac{1}{T} \int_0^T \tbf{1}(\h M_s(\theta_1) < -K) \d s = \exp(-\e^{\f{c} + K}) \leq \frac{\eps}{2}
  \end{equation*}
  and similarly
  \begin{equation*}
    \lim_{T \to \infty} \frac{1}{T} \int_0^T \tbf{1}(\h M_s(\theta_2) > K) \d s = 1 - \exp(-\e^{\f{c} - K}) \leq \frac{\eps}{2}.
  \end{equation*}
  Now \eqref{eq:liminf} follows from these bounds and \eqref{eq:union}.
\end{proof}

\section{Large excursions: Proof of the first part of Theorem~\ref{thm:LIL}.}
\label{sec:excursions}

In this section, we explain the proof of the first part of Theorem~\ref{thm:LIL}, which states that particles cross the radial threshold $b_t$ defined in \eqref{eq:threshold} infinitely often through every nonempty open sector $U \subset S^{d-1}$, provided $\beta_t$ satisfies \eqref{eq:integral-condition}.
Throughout the remainder of the paper, we assume $\beta$ satisfies the regularity condition in Theorem~\ref{thm:LIL}, which we write as
\begin{equation}
  \label{eq:beta-dot}
  \abss{\dot{\beta}_t} \ll t^{-1}.
\end{equation}
Integrating, this implies that
\begin{equation}
  \label{eq:beta-bounds}
  \abs{\beta_t} \ll \log t.
\end{equation}

\subsection{The Fisher--KPP equation}

We first obtain a Fisher--KPP representation of these excursions.
For the moment, suppose the set $U$ is deterministic; we treat random $U$ at the very end of the section.
Fix times $T<S$ and define the event
\begin{equation*}
  E_{T,S}(U) \coloneqq \big\{\exists \tau \in [T,S],  p \in \m{P}_\tau : R_\tau(p) = b_\tau, \Theta_\tau(p) \in U\big\}.
\end{equation*}
We are interested in the union
\begin{equation}
  \label{eq:event}
  E_T(U) \coloneqq \bigcup_{S\ge T}E_{T,S}(U),
\end{equation}
which captures excursions after time $T$.
We can write its probability as
\begin{equation}
  \label{26apr2310}
  \bar h_T(t,x)\coloneqq \P_{t,x}[E_T(U)]=\lim_{S\to\infty}\P_{t,x}[E_{T,S}(U)],
\end{equation}
where $\P_{t,x}$ denotes the law of a BBM started from $(t,x)$.
We also define the set 
\begin{equation}
  \label{26apr2404}
  \m{A}(t)= \{(R,\theta)\in\R^d: R=b_t, \theta \in U\},
\end{equation}
with the convention that
\begin{equation}
  \label{26apr2306}
  \m{A}(t)=\emptyset \For t<T.
\end{equation}
\begin{lemma}
  \label{lem:ancientbis}
  The function $\bar h_T(t,x)$ from \eqref{26apr2310} is the minimal positive solution to the global-in-time backward problem
  \begin{equation}
    \label{eq:ancientbis}
    \begin{cases}
      \partial_t\bar h_T+\Delta \bar h_T+\bar h_T-\bar h_T^2=0, & t\in\R, x\in\m{A}^\cc(t),\\
      \bar h_T(t,x)=1, & t>T, x\in\m{A}(t).
    \end{cases}
  \end{equation}
\end{lemma}
\noindent
This is similar to Lemma~A.1 in~\cite{Maillard13} in the one-dimensional setting.
\begin{proof}
  Given $t < S$ and $x\in \m{A}^\cc(t)$, let 
  \begin{equation}
    \label{26apr2308}
    h_{T,S}(t,x):=\P_{t,x}[E_{T,S}(U)]
  \end{equation}
  be the probability that a BBM started at time $t$ and position $x$ crosses the threshold $\m{A}(\tau)$ at some time $\tau \in [T,S]$.
  Let
  \begin{equation}
    \label{26apr2304}
    z_{T,S}(t,x) \coloneqq 1-h_{T,S}(t,x)=\P_{t,x}[E_{T,S}^\cc(U)].
  \end{equation}
  The branching property implies the following renewal identity for any $\tau \in (t, S)$:
  \begin{align}
    z_{T,S}(t,x)&=\e^{-(\tau-t)}\E[z_{T,S}(\tau,x+B_{\tau-t})]+
                  \int_t^\tau \e^{-(s-t)}\E[z_{T,S}^2(s,x+B_{s-t})] \d s\nonumber\\
                &=\e^{(\tau -t)(\Delta - 1)}[z_{T,S}(\tau,\anon)](x)+\int_t^\tau \e^{(s-t)(\Delta-1)}[z_{T,S}^2(s,\anon)](x) \d s.\label{26apr1008bis}  
  \end{align}
  Here, $B_s$ is a Brownian motion with diffusivity $\sqrt{2}\op{Id}$ that starts at time $t < S$ and position $x\in\R^d$, and dies if it reaches the moving boundary $\m{A}$ during the time interval $[T,S]$.
  If  $t<T$, then at times $s \in (t,T)$, there is no killing.
  The propagator $\e^{s \Delta}$ in \eqref{26apr1008bis} includes these boundary conditions.

  The first term on the right side of \eqref{26apr1008bis} corresponds to the event that the initial particle does not branch in the time interval $[t,\tau]$, which has probability $\e^{-(\tau-t)}$.
  The second term accounts for a first branching at time $s\in[t,\tau]$.
  In this case, the descendants of each child must avoid the threshold, giving $z_{T,S}^2$ by independence.

  The identity \eqref{26apr1008bis} is the mild form of the backward-in-time Fisher--KPP equation
  \begin{equation*}
    \partial_t z_{T,S}+\Delta z_{T,S} -z_{T,S}+z_{T,S}^2=0, \quad t<S,\, x\in \m{A}^\cc(t).
  \end{equation*}
  In addition, we have the terminal condition
  \begin{equation*}
    z_{T,S}(S,x)=1 \For x \in \m{A}^\cc(S)
  \end{equation*}
  and the boundary condition 
  \begin{equation*}
    z_{T,S}(t,x)=0 \For T\le t\le S, |x|=b_t, \theta\in U.
  \end{equation*}
  This boundary condition disappears at time $T$, leaving $z_{T,S}$ to solve Fisher--KPP in the whole space when $t < T$.
  One can view this as a whole-space problem for~$t < T$ with terminal data $z_{T,S}(T,x)$ at $t = T$.
  Using \eqref{26apr2304}, $h_{T,S}$ now satisfies
  \begin{equation}
    \label{26apr1012bis}
    \begin{cases}
      \partial_t h_{T,S}+\Delta h_{T,S}+h_{T,S}-h_{T,S}^2=0 & t<S,~x\in\m{A}^\cc(t),\\
      h_{T,S}(S,x)=0, & x\in\m{A}^\cc(S),\\
      h_{T,S}(t,x)=1, & x\in\m{A}(t),~T<t<S.
    \end{cases}
  \end{equation}
  By its definition~\eqref{26apr2308}, the family $(h_{T,S})_{S \geq T}$ is increasing in $S.$
  By \eqref{26apr2310}, its limit as $S \to \infty$ is $\bar{h}_T(t, x)$.
  It follows from parabolic estimates that $\bar h_T(t,x)$ solves the global-in-time backward problem \eqref{eq:ancientbis}, recalling the convention \eqref{26apr2306}.
  Moreover, the terminal condition in \eqref{26apr1012bis} implies that for every $S > T,$ $h_{T,S}$ lies below every global-in-time positive solution of \eqref{eq:ancientbis}.
  It follows that the limit $\bar{h}_T$ is the least such solution.
\end{proof}

\subsection{Proof of the first part of Theorem~\ref{thm:LIL}}

Assuming \eqref{eq:integral-condition}, we wish to show that for any $T \geq 2$, the minimal positive entire solution of \eqref{eq:ancientbis} is
\begin{equation}
  \label{26apr2316}
  \bar h_T(t,x)\equiv 1.
\end{equation}
We find it convenient to  reverse time in \eqref{eq:ancientbis} by setting
\begin{equation*}
  h_T(t,x) \coloneqq \bar h_T(-t,x).
\end{equation*}
This is an ancient solution of the forward-in-time Fisher--KPP problem
\begin{equation}
  \label{26apr2312}
  \begin{cases}
    \partial_th_T=\Delta h_T+h_T-h_T^2, & t\in\R, x\in\m{A}^\cc(-t)\\
    h_T(t,x)=1, & t<-T,~x\in\m{A}(-t).
  \end{cases}
\end{equation}
The key estimate in the proof of \eqref{26apr2316} states that far in the past, $h \gtrsim x\e^{-x}$ relative to $b_{\abs{t}}$.
\begin{proposition}
  \label{prop:der-mart-multi-bis}
  Let $\beta_t$ satisfy \eqref{eq:integral-condition}.
  Then for every nonempty open $U' \Subset U$, there exists $\tmax(d,U') \leq -T$ and $c(d,U') > 0$ such that
  \begin{equation}
    \label{26apr2325}
    h_T(t, r, \theta) \geq c [b_{\abs{t}} - r] \e^{-[b_{\abs{t}} - r]}
  \end{equation}
  for all $t \leq \tmax$, $\theta \in U'$, and $b_{\abs{t}} - r \in [4 \log \abs{t}, \abs{t}/8]$.
\end{proposition}
\noindent
The proof of Proposition~\ref{prop:der-mart-multi-bis} is rather technical, and constitutes Sections~\ref{sec:1D-excursions} and \ref{sec:excursions-proof}.

Here, we use a result of \cite{BerKimLubMalZei24} to show that Proposition~\ref{prop:der-mart-multi-bis} forces $h_T$ to be $1$.
The key element of the estimate \eqref{26apr2325} is the linear prefactor $b_{|t|}-r$, which enables a connection with the approximate derivative martingale $Y_t$.
\begin{proposition}
  \label{prop:hitting}
  For all $T \geq 2$, $h_T \equiv 1$.
\end{proposition}
\begin{proof}
  By definition $0 \leq h_T \leq 1$, so 
  the strong maximum principle for the Fisher--KPP equation implies that it suffices to show that 
  \begin{equation}
    \label{26apr2318}
    h_T(0,0)=1.
  \end{equation}
  Indeed, this would imply that $h_T(t,x)=1$ for all $t<0$ and $x\in\R^d$, which would immediately entail that $h_T(t,x)=1$ for all $t>0$ as well.

  Define the ``windowed'' BBM population
  \begin{equation}
    \label{eq:windowbis}
    \m{P}_t^{\textnormal{win}} \coloneqq \{p \in \m{P}_t : t^{1/6} \leq 2t - R_t(p) \leq t^{2/3}\}
  \end{equation}
  and the windowed derivative martingale
  \begin{equation}
    \label{eq:Y-window}
    Y_t^{\textnormal{win}} \coloneqq \left(\frac{\sqrt{\pi}}{t}\right)^{\frac{d-1}{2}} \sum_{p \in \m{P}_t^{\textnormal{win}}} [2t - R_t(p)] \e^{-[2t - R_t(p)]} \delta_{\Theta_p(t)}.
  \end{equation}

  By McKean's formula, for all $t \geq 0$, we have the representation
  \begin{align}
    h_T(0, 0)& = 1 - \E_{-t,0}\prod_{p \in \m{P}_0} \big[1 - h_T\bigl(-t, X_0(p)\bigr)\big]=1 - \E_{0,0}\prod_{p \in \m{P}_t} \big[1 - h_T\bigl(-t, X_t(p)\bigr)\big]\nonumber\\
             & \geq 1 - \E \exp\Big(\sum_{p \in \m{P}_t^{\textnormal{win}}} \log\big[1 - h_T\bigl(-t, X_t(p)\bigr)\big]\Big).\label{eq:McKean-1}
  \end{align}
  Here we use the shorthand $\E$ for $\E_{0,0}$.
  It follows from \eqref{eq:threshold} and \eqref{eq:beta-dot} that there exists $T_1$ such that for all $\abs{t} > T_1$,
  \begin{equation}
    \label{26apr2321}
    b_{\abs{t}} - r \in [4 \log \abs{t},\abs{t}/8]\quad \text{if} \enspace 2\abs{t} - r \in [\abs{t}^{1/6},\abs{t}^{2/3}].
  \end{equation}
  Fix a nonempty open $U' \Subset U$.
  Then Proposition~\ref{prop:der-mart-multi-bis}, \eqref{eq:McKean-1}, and \eqref{26apr2321} yield
  \begin{equation}
    \label{26apr2322}
    h_T(0, 0) \geq 1 - \E \exp\Big(\sum_{p \in \m{P}_t^{\textnormal{win}}} 
    \log\big\{1 - c [b_t - R_t(p)] \e^{-[b_t - R_t(p)]} \tbf{1}_{U'}\bigl(\Theta_p(t)\bigr)\big\}\Big)
  \end{equation}
  for all $t\ge \max\{-\tmax(d,U'),T_1\}$.

  Note that, simply from \eqref{eq:threshold}, we can re-write the exponential in \eqref{26apr2322} as  
  \begin{equation}
    \label{26apr2323}
    \e^{-[b_t - R_t(p)]} = t^{-(d-2)/2} \e^{-\beta_t} \e^{-[2t - R_t(p)]}.
  \end{equation}
  For the windowed population, $2t - R_t(p) \ge t^{1/6}$    while $\abs{b_t - 2t} \lesssim \log t$.
  It follows that the prefactor in \eqref{26apr2322} can be bounded as 
  \begin{equation}
    \label{26apr2324}
    b_t - R_t(p) \geq \frac{2t - R_t(p)}{2} \For p\in\m{P}_t^{\textnormal{win}}
  \end{equation}
  for $t$ sufficiently large.
  Using \eqref{26apr2323} and \eqref{26apr2324} in \eqref{26apr2322} shows that $h_T(0,0)$ is bounded from below by
  \begin{equation*}
    1- \E \exp\Big(\sum_{p \in \m{P}_t^{\textnormal{win}}} \log\big[1 - \tfrac{c}{2} t^{-\frac{d-2}{2}} \e^{-\beta_t} [2t - R_t(p)] \e^{-[2t - R_t(p)]}\tbf{1}_{U'}\bigl(\Theta_p(t)\bigr)\big]\Big).
  \end{equation*}
  Recall the windowed version $Y_t^{\textnormal{win}}$ of the derivative martingale defined in \eqref{eq:shaved}.
  Because $\log(1 - a) \leq -a$, we obtain
  \begin{align}
    h_T(0,0) &\geq 1 - \E \exp\Big(-\tfrac{c}{2} t^{-\frac{d-2}{2}} \e^{-\beta_t} \sum_{p \in \m{P}_t^{\textnormal{win}}} [2t - R_t(p)] \e^{-[2t - R_t(p)]} \tbf{1}_{U'}\bigl(\Theta_p(t)\bigr)\Big)\nonumber\\
             &\geq 1 - \E \exp\Big[-\frac{c}{2\pi^{(d-1)/4}} t^{1/2} \e^{-\beta_t} Y_t^{\textnormal{win}}(U')\Big].\label{eq:der-mart-bis}
  \end{align}
  for all sufficiently large $t$.
  
  Now, Theorem~1.4 of \cite{BerKimLubMalZei24} states that $Y_t^{\textnormal{win}} \to Z_\infty$ in probability in the vague topology as $t \to \infty$.
  As noted in the proof of Theorem~1.1, it follows that $Y_t^{\textnormal{win}} \to Z_\infty$ almost surely along a subsequence.
  Also, Theorem~1.3 of \cite{StaBerMal21} ensures that $Z_\infty$ is almost surely positive on nonempty open sets (the theorem states that $Z_\infty(S^{d-1}) > 0$, but the proof implies $Z_\infty(U) > 0$ for any open $U$).
  Combining these results, the portmanteau lemma implies that along the subsequence,
  \begin{equation*}
    \liminf_{t \to \infty} Y_t^{\textnormal{win}}(U') \geq Z_\infty(U') > 0 \quad \text{a.s.}
  \end{equation*}
  By \eqref{eq:beta-bounds}, $t^{1/2} \e^{-\beta_t} \to \infty$ as $t \to \infty$.
  It follows that, still along the subsequence,
  \begin{equation*}
    \liminf_{t \to \infty} \frac{c}{2\pi^{(d-1)/4}} t^{1/2} \e^{-\beta_t} Y_t^{\textnormal{win}}(U') = \infty \quad \text{a.s.}
  \end{equation*}
  Sending $t \to \infty$ in the right side of \eqref{eq:der-mart-bis}, we see that $h_T(0,0) \geq 1.$
  Of course we also know that $h_T \leq 1$, whence $h_T(0,0) = 1$, which is \eqref{26apr2318}.
\end{proof}
\begin{remark}
  \label{rem:3D}
  We can now see why the ``subcritical'' regime $d \geq 3$ and $p > \frac{d-1}{d-2}$ permits a much shorter proof.
  As discussed after \eqref{eq:informal}, subcriticality allows us to use $\beta_t = -\log\log t$ in the threshold \eqref{eq:threshold}.

  This relaxation allows us to omit the linear prefactor in \eqref{26apr2325}.
  Indeed, given a pure exponential lower bound in \eqref{26apr2325}, we would obtain an analogue $A_t$ of an \emph{additive} martingale (as in \eqref{eq:additive} below) rather than the \emph{derivative} martingale $Y_t$ appearing above.
  The Seneta--Heyde scaling~\cite{HuShi09,AidShi14} implies that $t^{1/2} A_t \to CZ_\infty$, so the bracketed expression in \eqref{eq:der-mart-bis} is then of order $\e^{-\beta_t}$.
  This diverges if $\beta_t = -\log \log t$, and the proof of Proposition~\ref{prop:hitting} goes through.
  It is much easier to prove \eqref{26apr2325} without the linear prefactor, as one can see comparing the proofs of Lemmas~\ref{lem:exp-lower} and \ref{lem:der-mart-lower} below.
  
  Conversely, to treat the critical regime such as $d = 2$, we must use a more stringent threshold with $\beta_t = +\log\log t$.
  This requires the linear prefactor in \eqref{26apr2325} to overcome the decaying factor $\e^{-\beta_t}$ in \eqref{eq:der-mart-bis}, and the prefactor demands a considerably more sophisticated argument.
\end{remark}
We can now prove the positive part of Theorem~\ref{thm:LIL}.
\begin{proof}[Proof of first part of Theorem~\ref{thm:LIL}]
  Assume $\beta_t$ satisfies \eqref{eq:integral-condition} and \eqref{eq:beta-dot}.
  By \eqref{26apr2310} and Proposition~\ref{prop:hitting},
  \begin{equation*}
    \P\Big[\bigcap_{n \in \N} E_n(U)\Big] = 1
  \end{equation*}
  for $E_T(U)$ defined in \eqref{eq:event}.
  Working in this intersection, let $t_n$ be the earliest time contributing to $E_n(U)$, with corresponding particle $p_n \in \m{P}_{t_n}$.
  The random sequence $(t_n)_{n \in \N}$ is increasing and satisfies $t_n \geq n$, so $t_n \nearrow \infty$.
  Moreover, \eqref{eq:large-excursion} holds by the definition of $E_n(U)$.
  For random $U$, follow the same argument as in the proof of Theorem~\ref{thm:main}.
\end{proof}

\section{One-dimensional analysis}
\label{sec:1D-excursions}

To prove Proposition~\ref{prop:der-mart-multi-bis}, we construct a suitable subsolution to the Fisher--KPP problem \eqref{26apr2312}.
This anisotropic multi-dimensional subsolution uses 1D subsolutions as building blocks.
In this section, we construct these 1D subsolutions, which incidentally yield a largely PDE-based proof of one direction of~\cite[Theorem~1.1]{Hu15}.

Taking $d = 1$ in \eqref{eq:threshold}, we define
\begin{equation}
  \label{eq:1D-barrier}
  b_t^{\textnormal{1D}} \coloneqq 2t - \frac{1}{2} \log t + \beta_t.
\end{equation}
We note for future use that \eqref{eq:beta-dot} yields
\begin{equation}
  \label{eq:drift-sign}
  \dot{b}_t^{\textnormal{1D}} < 2.
\end{equation}
for $t$ sufficiently large.

We study a 1D analogue of \eqref{26apr2312} in the ${b}_t^{\textnormal{1D}}$-moving frame:
\begin{equation}
  \label{eq:ancient}
  \begin{cases}
    \partial_t u = \partial_x^2 u + \dot{b}_{\abs{t}}^{\textnormal{1D}} \partial_x u + u - u^2,& t < 0,~x>0,\\
    u(t, 0) = 1.
  \end{cases}
\end{equation}
We will prove the following one-dimensional analog of Proposition~\ref{prop:der-mart-multi-bis}.
\begin{proposition}
  \label{prop:der-mart-multi-bis2}
  There exists $\tmax\le -1$ and a constant $c > 0$ such that 
  \begin{equation}
    \label{26apr2326}
    u(t, x) \geq c x\e^{-x}
  \end{equation}
  for all $t \leq \tmax$ and $x\in [4 \log \abs{t}, \abs{t}/8]$.
\end{proposition}
For our application, the subsolutions underlying Proposition~\ref{prop:der-mart-multi-bis2} are more valuable than the result itself, but the proposition serves as a useful guidepost.
Moreover, in combination with Proposition~\ref{prop:hitting}, it yields a new proof (in the BBM setting) of the divergent-integral direction of Hu's 1D integral test~\cite[Theorem~1.1]{Hu15}.
\medskip

The first step in the proof of Proposition~\ref{prop:der-mart-multi-bis2} is establishing the exponential decay in \eqref{26apr2326} without the (quite important) linear prefactor.
We express this through a subsolution satisfying certain properties, which we use later in the proof of the multi-dimensional Proposition~\ref{prop:der-mart-multi-bis}.
\begin{lemma}
  \label{lem:exp-lower}
  There exists a constant $\cexp \in (0, 1]$ such that for each $\f{t} \leq -1$, there exists a function $\uconst(t,x; \f{t}) \geq 0$ satisfying the following properties:
  \begin{enumerate}[label = \textnormal{(\roman*)},itemsep = 2pt]
  \item $\uconst = 0$ where $t \ll -1$ or $x \geq \abs{t}$;

  \item
    \label{item:exp-sub}
    $\uconst$ is a subsolution of \eqref{eq:ancient};

  \item $\partial_x \mr{u} \leq 0$; and

  \item for all $x \in [0, \abs{\f{t}}/2]$,
    \begin{equation}
      \label{26apr1320_1}
      \uconst(\f{t},x; \f{t}) \geq \cexp \e^{-x}.
    \end{equation}
  \end{enumerate}
\end{lemma}
It follows that any positive solution of \eqref{eq:ancient} is bounded from below by $\cexp \e^{-x}$ in an appropriate region.
This is progress, but not the final bound we seek.
\begin{proof}
  Given $T \gg 1,$ let $\phi$ solve the initial-boundary value problem
  \begin{equation}
    \label{eq:uconst}
    \begin{cases}
      \partial_t \phi = \partial_x^2 \phi + 2 \partial_x \phi + \phi - \phi^2 & \text{for }x>0,~t>-T,\\
      \phi(t, 0) = 1 & \text{for } t > -T,\\
      \phi(-T,x) = 0& \text{for } x > 0.
    \end{cases}
  \end{equation}
  The initial condition $0$ is a subsolution, so $\phi$ is increasing in $t$.
  By time translation symmetry, $\phi$ thus converges locally uniformly as $t+T \to \infty$ to the minimal positive solution $\bar\phi$ of the steady-state ODE
  \begin{equation*}
    \bar\phi'' + 2 \bar\phi' + \bar\phi - \bar\phi^2 = 0 \enspace\text{in } \R_+ \text{ with } \enspace \bar\phi(0) = 1.
  \end{equation*}

  In addition, as $\e^{-x}$ is a supersolution of \eqref{eq:uconst}, we have $\phi \leq \e^{-x}$.
  It follows that 
  \begin{equation}
    \label{26apr1304}
    \bar\phi(x)\leq \e^{-x}
  \end{equation}
  as well.
  Writing $\bar\phi(x)=\e^{-x} v$, the function $v$ solves the ODE
  \begin{equation}
    \label{26apr1310}
    v'' = \e^{-x}v^2 \enspace\text{in } \R_+ \text{ with } \enspace v(0) = 1.
  \end{equation}
  We see from \eqref{26apr1304} that 
  \begin{equation}
    \label{26apr1314}
    v \leq 1,
  \end{equation}
  so $v$ is convex and bounded.
  Thus, the limit $v(\infty) \geq 0$ exists, $v'(\infty) = 0$, and $v'(x) < 0$ for all $x > 0$.
  We claim that 
  \begin{equation*}
    v(\infty) > 0.
  \end{equation*} 
  Suppose, for the sake of contradiction, that $v(\infty) = 0$.
  Integrating \eqref{26apr1310}, we find
  \begin{equation}
    \label{26apr1312}
    -v'(x) = \int_x^\infty \e^{-y}v^2(y) \d y \leq v^2(x) \int_x^\infty \e^{-y} \d y = \e^{-x} v^2(x).
  \end{equation}
  By the same reasoning, integrating \eqref{26apr1312} and using \eqref{26apr1314}, we see that if $v(\infty) = 0$, then 
  \begin{equation*}
    0 \leq v(x) \leq \e^{-x} v^2(x) \leq \e^{-x}v(x) \ForAll x > 0.
  \end{equation*}
  It follows that $v\equiv 0$, violating the boundary condition $v(0) = 1$.
  Through this contradiction, we see that there exists some constant $c > 0$ so that 
  $v(x) \geq c$ for all~$x>0$.
  
  Now, fix $\f{t} \leq -1$.
  Because 
  \begin{equation*}
    \phi(t+T,x) \to \e^{-x}v (x)\geq c\e^{-x} \As t+T\to+\infty,
  \end{equation*}
  locally uniformly in $t$ and $x$, 
  we can take $T \gg 1$ so that 
  \begin{equation}
    \label{26apr1318}
    \phi(\f{t},x) \geq (c/2) \e^{-x} \ForAll x \leq \abs{\f{t}}.
  \end{equation}
  It just remains to cut off $\phi$ where $x \geq \abs{t}$, which will prove useful in multi-dimensional applications.

  To this end, we note that $V(t,x) \coloneqq \e^x \phi(t,x) \leq 1$ satisfies 
  \begin{equation}
    \label{26apr1316}
    \partial_t V = \partial_x^2 V - \e^{-x} V^2, \quad x>0,\enspace t>-T.
  \end{equation}
  Since $\partial_t\phi>0$, we also have $\partial_t V > 0$.
  Therefore
  \begin{equation*}
    \partial_x^2 V = \partial_t V + \e^{-x} V^2 \geq 0.
  \end{equation*}
  Thus $V(t,x)$ is likewise convex in space and bounded, which implies 
  \begin{equation*}
    \partial_x V \leq 0.
  \end{equation*}
  Now let 
  \begin{equation*}
    m(t, x) \coloneqq 1 - \e^{x + t}\hbox{ for } t<0,x>0
  \end{equation*}
  This function is a solution to the heat equation that also satisfies $m(t,-t)=0$.
  Then, as $0\le m(t,x)\le 1$ for $x \leq \abs{t}$, we have
  \begin{equation*}
    (\partial_t - \partial_x^2)(mV) = -2 \partial_x m \partial_x V - m\e^{-x}V^2 
    \leq - m^2 \e^{-x} V^2 \For 0<x<|t|.
  \end{equation*}
  Thus, $mV$ is a subsolution of \eqref{26apr1316}.
  It follows that $m\phi$ is a subsolution of \eqref{eq:uconst} where $t \geq -T$ and $0 < x \leq \abs{t}$.
  We may extend $m\phi$ by zero to the region where $t \leq -T$ or $x \geq \abs{t}$, leading to a subsolution in the whole region $t>-T$ and $x>0$.
  Moreover, we have
  \begin{equation*}
    m(t,x) \geq 1 - \e^{-\abs{t}/2} \geq 1 - \e^{-1/2} \quad \text{where }x \leq \abs{t}/2.
  \end{equation*}
  Recalling also \eqref{26apr1318}, we see that
  \begin{equation}
    \label{26apr1321}
    m(\f{t},x)\phi(\f{t},x) \geq \cexp \e^{-x} \quad \text{where } x \leq \abs{\f{t}}/2,
  \end{equation}
  with
  \begin{equation*}
    \cexp \coloneqq \min\{c(1 - \e^{-1/2})/2, 1\} > 0.
  \end{equation*}
  We let 
  \begin{equation*}
    \mr{u}(t,x) \coloneqq m(t,x)\phi(t,x),
  \end{equation*}
  noting that this depends on $\f{t}$ through the choice of $T$ ensuring \eqref{26apr1318}.
  We extend $\mr{u}$ by zero before $t = -T$.
  
  The comparison principle implies that $\partial_x \phi(t,x) \leq 0$.
  Since $m(t,x)$ is also decreasing in $x$, the same is true for the product: $\partial_x \mr{u} \leq 0$.
  By \eqref{eq:drift-sign}, $\mr{u}$ is then a subsolution of \eqref{eq:ancient}.
  Moreover, $\mr{u}$ satisfies \eqref{26apr1320_1} because of \eqref{26apr1321}.
\end{proof}
Applying Lemma~\ref{lem:exp-lower} at $x = 2 \log \abs{t}$, we have
\begin{equation}
  \label{26apr1502}
  u(t,2\log|t|)\geq \cexp \abs{t}^{-2}.
\end{equation}
We now study \eqref{eq:ancient} to the right of $x = 2 \log \abs{t}$.
This is a convenient choice because after an exponential tilt, the nonlinearity will have a prefactor $\abs{t}^{-2}$, which is integrable in time.
This will allow us to easily absorb the nonlinearity.
The following lemma introduces the factor of $x$ that was missing in the lower bound \eqref{26apr1320_1} but is needed for Proposition~\ref{prop:der-mart-multi-bis2}.
It makes use of the constant $\cexp > 0$ from Lemma~\ref{lem:exp-lower}.
\begin{lemma}
  \label{lem:der-mart-lower}
  For all $\cubar \in (0, 1]$, there exists $\tmax \leq -1$ such that the following holds.
  For all $\f{t} \leq \tmax$, there exists $\ubar{T}(\f{t}) > \abs{\f{t}}$ and a function $\ubar{u}(t,x; \f{t}) \geq 0$ defined where $t \in [-\ubar{T}(\f{t}), \f{t}]$ and $x \geq 2 \log \abs{t}$ satisfying the following properties:
  \begin{enumerate}[label = \textnormal{(\roman*)},itemsep = 2pt]
  \item $\ubar{u} = 0$ where $x \geq \abs{t}$;

  \item
    \label{item:der-mart-upper}
    $\ubar{u} \leq \cubar \e^{-x}$ where $t = -\ubar{T}$ or $x = 2 \log \abs{t}$;   

  \item $\ubar{u}$ is a subsolution of \eqref{eq:ancient};

  \item $\partial_x \ubar{u} \leq 0$; and

  \item for all $x \in [4 \log \abs{\f{t}}, \abs{\f{t}}/4]$,
    \begin{equation}
      \label{26apr1320}
      \ubar{u}(\f{t}, x; \f{t}) \geq 2^{-8} \cubar x \e^{-x}.
    \end{equation}
  \end{enumerate}
\end{lemma}
Returning to the solution $u$ of \eqref{eq:ancient}, the boundary estimate \eqref{26apr1502}, Lemma \ref{lem:der-mart-lower}, and the comparison principle imply that
\begin{equation*}
  u(t,x) \geq 2^{-8} \cubar x \e^{-x} \ForAll x \in [4 \log \abs{t}, \abs{t}/4] \hbox{ and $t\le \tmax$}.
\end{equation*}
This is \eqref{26apr2326}, so Proposition~\ref{prop:der-mart-multi-bis2} reduces to Lemma~\ref{lem:der-mart-lower}.
This lemma will also play a key role in the proof of Proposition~\ref{prop:der-mart-multi-bis} below.
\begin{proof}
  Let $u$ be a nonnegative solution of
  \begin{equation*}
    \partial_t u = \partial_x^2 u + \dot{b}_{\abs{t}}^{\textnormal{1D}} \partial_x u + u - u^2
  \end{equation*}
  that satisfies a version of \eqref{26apr1502}:
  \begin{equation*}
    u(t,2\log |t|)\ge \cubar |t|^{-2}.
  \end{equation*}
  We write
  \begin{equation}
    \label{eq:1D-transform}
    u(t,x) \coloneqq \e^{-x} v(t,x) \And v(t, x) \coloneqq w\bigl(t, x - 2 \log \abs{t}\bigr).
  \end{equation}
  Then the function $w$ satisfies
  \begin{equation}
    \label{eq:shift}
    \begin{aligned}
      &\partial_t w = \partial_x^2 w + \Big[\frac{1}{2\abs{t}} - \dot{\beta}_{\abs{t}}\Big] w 
        - \kappa(t) \partial_x w - \abs{t}^{-2} \e^{-x} w^2,\quad x>0,\\
      & w(t, 0) \geq \cubar
    \end{aligned}
  \end{equation}
  with
  \begin{equation}
    \label{26apr1602}
    \kappa(t) \coloneqq \frac{5}{2 \abs{t}} - \dot\beta_{\abs{t}} = \m{O}(\abs{t}^{-1}).
  \end{equation}
  Using \eqref{eq:beta-dot}, we may assume that $t \leq \tmax$ is sufficiently negative that
  \begin{equation}
    \label{eq:drift-signs}
    \frac{1}{2 \abs{t}} - \dot\beta_{\abs{t}} > 0 \enspace \text{and hence} \enspace \kappa > 0.
  \end{equation}
  We will construct a sub-solution $\ubar{w}$ of \eqref{eq:shift} of the form 
  \begin{equation}
    \label{26apr2016}
    \ubar{w}=a(t)(z + 1),
  \end{equation}
  for a certain multiplier~$a(t) \leq \cubar$ and a solution $z(t,x)$ of a linear problem.
  If we solve the linearization of \eqref{eq:shift} using $\cubar(\ti z + 1)$, we find that
  \begin{equation}
    \label{eq:z-trial}
    \begin{aligned}
      &\partial_t \ti z = \partial_x^2 \ti z + \Big[\frac{1}{2\abs{t}} 
        - \dot{\beta}_{\abs{t}}\Big](\ti z + 1) - \kappa(t) \partial_x \ti z, \quad x>0,\\
      & \ti z|_{x = 0} \geq 0.
    \end{aligned}
  \end{equation}
  We construct a subsolution by reducing $\ti z$ in several ways.
  We set $\ti z = 0$ where $x = 0$ and at $t = -T$ for $T \gg 1$.
  We also reduce its bulk forcing.
  Let us set
  \begin{equation}
    \label{eq:erf}
    \Phi(x) \coloneqq \pi^{-1/2} \int_0^x \e^{-y^2/4} \d y,
  \end{equation}
  which satisfies $\Phi(0) = 0$ and $\Phi \nearrow 1$ as $x \to \infty$.
  We replace $1$ by $\Phi$ in \eqref{eq:z-trial}, which has the advantage that it smoothly achieves the boundary data.
  After these changes, we are left with the following problem:
  \begin{equation}
    \label{eq:finite-time}
    \begin{cases}
      \partial_t z = \partial_x^2 z + \Big[\frac{1}{2\abs{t}} 
      - \dot{\beta}_{\abs{t}}\Big][z + \Phi(x)] 
      - \kappa(t) \partial_x z & \text{for } t \in (-T,0), x > 0,\\
      z(t,0) = 0 & \text{for } t \in (-T,0),\\
      z(-T,x) = 0 & \text{for } x > 0.
    \end{cases}
  \end{equation}
  This is a subsolution for \eqref{eq:z-trial}.
  Our next goal is the following lemma.
  \begin{lemma}
    \label{lem-26apr2002}
    Suppose there exists $C>0$ so that 
    \begin{equation}
      \label{26may516}
      \kappa(t)\le \frac{C}{|t|} \ForAll t \leq -1.
    \end{equation}
    Then there exists $\tmax \leq -1$ such that for all $\f{t} \leq \tmax$, there exists $T=T(\f{t}) > |\f{t}|$ such that the solution $z(t,x)$ of \eqref{eq:finite-time} satisfies
    \begin{equation}
      \label{eq:lowerbis}
      z(\f{t}, x) \geq x \ForAll x \leq \frac{\abs{\f{t}}}{2}
    \end{equation}
    and
    \begin{equation}
      \label{eq:z-derivbis}
      \abs{\partial_x z(t,x)} \leq 16 \ForAll t<\f{t}, x > 0.
    \end{equation}
  \end{lemma}
  The general condition \eqref{26may516} will allow us to treat a slightly different problem in a subsequent section.
  \begin{proof}
    Let 
    \begin{equation}
      \label{26apr1620}
      \lambda(t) \coloneqq \abs{t}^{-1/2} \e^{\beta_{\abs{t}}},
    \end{equation}
    which satisfies
    \begin{equation*}
      \dot\lambda(t) = \Big[\frac{1}{2 \abs{t}} - \dot{\beta}_{\abs{t}}\Big] \lambda(t).
    \end{equation*}
    We use Duhamel's formula to express $z$ in terms of an initial-value problem started at $s \in [-T,t]$ without the zeroth-order term in \eqref{eq:finite-time}.
    Writing
    \begin{equation*}
      z(t,x)=\lambda(t)r(t,x),
    \end{equation*}
    the function $r(t,x)$ satisfies
    \begin{equation*}
      \partial_tr=\partial_x^2r -\kappa(t)\partial_x r + \frac{\dot\lambda(t)}{\lambda^2(t)} \Phi(x).
    \end{equation*}
    Let $\zeta(t, x; s)$ solve
    \begin{equation}
      \label{eq:Duhamel-piecebis}
      \begin{cases}
        \partial_t \zeta = \partial_x^2 \zeta - \kappa(t) \partial_x \zeta & \text{for } t>s,x>0,\\
        \zeta(t,0;s) = 0 & \text{for } t > s,\\
        \zeta(s,x; s) = \Phi & \text{for } x>0.
      \end{cases}
    \end{equation}
    Then the function $r$ is given by Duhamel's formula
    \begin{equation*}
      r(t, x) = \int_{-T}^t \frac{\dot\lambda(s)}{\lambda^2(s)}  \zeta(t, x; s) \d s,
    \end{equation*}
    and $z(t,x)$ is given by
    \begin{align}
      z(t, x) &= \lambda(t) \int_{-T}^t \frac{\dot\lambda(s)}{\lambda^2(s)}  \zeta(t, x; s) \d s\nonumber\\
              &= \lambda(t) \int_{-T}^t \abs{s}^{1/2} \e^{-\beta_{\abs{s}}} \left[\frac{1}{2 \abs{s}} - \dot{\beta}_{\abs{s}}\right] \zeta(t, x; s) \d s.\label{eq:Duhamelbis2}
    \end{align}
    In the absence of drift ($\kappa=0$), the solution of \eqref{eq:Duhamel-piecebis} is simply
    \begin{equation*}
      \zeta_*(t,x;s) \coloneqq \Phi\bigl(x/\sqrt{t-s+1}\bigr).
    \end{equation*}
    Even when $\kappa \neq 0$, because $\Phi$ is increasing, the comparison principle implies that $\partial_x \zeta > 0$ at all times.
    It follows from \eqref{eq:drift-signs} that $\kappa \partial_x \zeta > 0$, so the drift in \eqref{eq:Duhamel-piecebis} has a negative effect.
    By comparison,
    \begin{equation}
      \label{eq:one-side}
      \zeta(t,x;s)\leq \zeta_*(t,x;s).
    \end{equation}
    However, we will need the result of Lemma~\ref{lem-26apr2002} with a more general $\kappa(t)$ in Section~\ref{sec:necessary} below.
    We therefore only assume that $\kappa(t)$ satisfies \eqref{26may516}, and do not rely on \eqref{eq:one-side}.

    We will argue that the drift does not have a significant effect on $\zeta$, so we aim to show that the difference $\zeta - \zeta_*$ is small.
    To accomplish this, we pass to self-similar variables: define
    \begin{equation*}
      \tau \coloneqq \log(t - s + 1) \And \eta \coloneqq x(t - s + 1)^{-1/2}.
    \end{equation*}
    It suffices to treat $t \in [s, -1]$, which corresponds to $\tau \leq [0,\log\abs{s}]$.
    Write
    \begin{equation}
      \label{eq:original-to-ss}
      \zeta(t,x;s) = \zeta_*(t,x;s) + q\bigl(\log(t - s + 1), x(t - s + 1)^{-1/2}\bigr),
    \end{equation}
    with an ``error'' term $q$.
    Then $q$ satisfies
    \begin{equation*}
      \partial_\tau q-\frac{\eta}{2}\partial_\eta q= \partial_\eta^2q - \kappa(t)\e^{\tau/2}\Phi'(\eta)-\kappa(t)\e^{\tau/2}\partial_\eta q.
    \end{equation*}
    Next, we write $q$ as
    \begin{equation}
      \label{26apr1614}
      q(\tau,\eta)=\e^{-\eta^2/8}p(\tau,\eta).
    \end{equation}
    A calculation shows that $p$ satisfies
    \begin{equation}
      \label{eq:ss}
      \begin{cases}
        \partial_\tau p = -\m{M} p - g \big(\partial_\eta - \frac{\eta}{4}\big)p 
        - g \pi^{-1/2} \e^{-\eta^2/8} & \text{for } \tau \in (0, \log\abs{s}), \eta > 0,\\
        p(\tau,0) = 0 & \text{for } \tau \in (0,\log \abs{s}),\\
        p(0,\eta) = 0 & \text{for } \eta > 0
      \end{cases}
    \end{equation}
    for the operator
    \begin{equation}
      \label{26apr1606}
      \m{M} \coloneqq -\partial_\eta^2 + \frac{\eta^2}{16} + \frac{1}{4}
    \end{equation}
    and the function
    \begin{equation*}
      g(\tau, s) \coloneqq \e^{\tau/2} \kappa\left(\e^\tau + s - 1\right),      
    \end{equation*}
    which satisfies
    \begin{equation}
      \label{eq:g-est}
      |g(\tau,s)|\lesssim \frac{\e^{\tau/2}}{1 + \abs{s} - \e^\tau}.
    \end{equation}

    We show that $p$ is small through a standard energy estimate, multiplying \eqref{eq:ss} by $p$ and integrating.
    In the following, $\norm{\anon}$ and $\braket{\anon,\anon}$ denote the $L^2(\R_+)$ norm and inner product, respectively.
    We find
    \begin{equation}
      \label{eq:energy-1}
      \frac{1}{2} \der{}{\tau} \norm{p}^2 = -\braket{\m{M}p,p} + 
      \frac{g}{4} \braket{\eta p,p} + g \pi^{-1/2}\braket{\e^{-\eta^2/8},p}.
    \end{equation}
    Young's inequality and \eqref{26apr1606} imply that
    \begin{equation*}
      \frac{|g|}{4} \braket{\eta p,p} \leq \braket{(\eta^2/16 + g^2)p,p} \leq \braket{\m{M}p,p} + g^2 \norm{p}^2.
    \end{equation*}
    Similarly, we have 
    \begin{equation*}
      |\braket{\e^{-\eta^2/8},p} |\leq \norm{p}^2 + \frac{\sqrt{\pi}}{4}.
    \end{equation*}
    Thus \eqref{eq:energy-1} yields
    \begin{equation}
      \label{eq:energy-2}
      \frac{1}{2} \der{}{\tau} \norm{p}^2 \leq (g^2 + \abs{g}) \norm{p}^2 + \abs{g}.
    \end{equation}
    Recalling that $\tau \leq \log \abs{s}$, we use \eqref{eq:g-est} to compute
    \begin{equation}
      \label{26apr1610}
      \int_0^{\log \abs{s}} g^2(\tau,s) \d \tau \lesssim \int_0^{\log \abs{s}} \frac{\e^\tau}{(1 + \abs{s} - \e^\tau)^2} \d \tau = \int_0^{\abs{s}} \frac{\dn r}{(1 + \abs{s} - r)^2} \leq 1
    \end{equation}
    and
    \begin{align}
      \int_0^{\log \abs{s}} \abs{g(\tau,s)} \d \tau  &\lesssim \int_0^{\log \abs{s}} \frac{\e^{\tau/2}}{1 + \abs{s} - \e^\tau} \d \tau 
                                                       \lesssim \int_0^{\abs{s}^{1/2}} \frac{\dn r}{1 + \abs{s} - r^2}\nonumber\\
                                                     &\lesssim \abs{s}^{-1/2} \int_0^{\abs{s}^{1/2}} \frac{\dn r}{\sqrt{\abs{s} + 1} - r} \lesssim \frac{\log \abs{s}}{\abs{s}^{1/2}}.\label{26apr1612}
    \end{align}
    Applying Gr\"onwall to \eqref{eq:energy-2}, \eqref{26apr1610} and \eqref{26apr1612} yield
    \begin{equation*}
      \norm{p}^2(\tau) \lesssim \int_0^{\tau} \abs{g(\tau', s)} \d \tau' \lesssim \frac{\log \abs{s}}{\abs{s}^{1/2}}.
    \end{equation*}
    Parabolic estimates allow us to upgrade this $L^2$ estimate to one in $W^{1,\infty}$:
    \begin{equation}
      \label{eq:p-small}
      \sup_{\eta > 0} (\abs{p} + \abss{\partial_\eta p}) \lesssim  \abs{s}^{-1/8}.
    \end{equation}
    Indeed, traditional parabolic estimates hold on a large compact domain $[0,L]$, and the confining potential $\eta^2/16$ in $\m{M}$ implies a maximum principle that prevents local maxima outside $[0,L]$.
    Integrating the derivative estimate from the Dirichlet condition $p|_{\eta = 0} = 0$, we also find $p \lesssim \eta \abs{s}^{-1/8}$.
    In combination,
    \begin{equation}
      \label{26apr1618}
      p \lesssim (\eta \wedge 1) \abs{s}^{-1/8}.
    \end{equation}
    Recalling the transformations \eqref{eq:original-to-ss} and \eqref{26apr1614} that link $p$ to $\zeta$, \eqref{26apr1618} implies that
    \begin{equation}
      \label{eq:zeta-error}
      \frac{1}{2} \Phi\Big(\frac{x}{\sqrt{t - s + 1}}\Big) 
      \leq \zeta(t,x;s) \leq 2 \Phi\Big(\frac{x}{\sqrt{t - s + 1}}\Big), 
    \end{equation}
    provided $s \leq t \leq \tmax$ is sufficiently negative.
    Using \eqref{26apr1620} and \eqref{eq:Duhamelbis2}, we obtain
    \begin{equation}
      \label{eq:equiv}
      \frac{1}{2} \leq \frac{z(t, x)}{\displaystyle\e^{\beta_{\abs{t}}} \abs{t}^{-1/2} \int_{-T}^t \abs{s}^{-1/2} \e^{-\beta_{\abs{s}}} \Phi\Big(\frac{x}{\sqrt{t - s + 1}}\Big) \d s} \leq 2.
    \end{equation}
    To bound $z$ more explicitly from above, we note that $\Phi(x) \leq \pi^{-1/2}x \leq x$.
    This gives 
    \begin{equation}
      \label{26apr1621}
      \begin{aligned}
        z(t, x) &\leq 2x \e^{\beta_{\abs{t}}} \abs{t}^{-1/2} 
                  \int_{-T}^t \e^{-\beta_{\abs{s}}} \abs{s}^{-1/2} (t - s + 1)^{-1/2} \d s.
      \end{aligned}
    \end{equation}
    Note that for $-T<s<t<-1$ we have 
    \begin{equation}
      \label{26apr2006}
      t-s+1=|s|-|t|+1<|s|.
    \end{equation}
    In addition, it follows from \eqref{eq:beta-dot} that there exists $T_0$ so that 
    \begin{equation}
      \label{26apr2608}
      \der{}{r}\big(r^{1/2}\e^{\beta_r}\big)=\Big(\frac{1}{2r^{1/2}}+r^{1/2}\dot\beta\Big)\e^{\beta_r} \geq 0 \ForAll r \geq T_0.
    \end{equation}
    Taking $\tmax \leq -T_0$, we deduce from \eqref{26apr1621} that
    \begin{align*}
      z(t,x) & \le 2 x \e^{\beta_{\abs{t}}} \abs{t}^{-1/2} \int_{-T}^t \e^{-\beta_{t-s+1}} (t-s+1)^{-1/2} (t - s + 1)^{-1/2} \d s.\\
             &\leq 2 x \e^{\beta_{\abs{t}}} \abs{t}^{-1/2} \int_1^{T - \abs{t} + 1} \frac{\dn r}{r \exp \beta_r}.
    \end{align*}
    Thus, if we let
    \begin{equation*}
      \m{I}(r) \coloneqq \int_1^r \frac{\dn r'}{r' \exp \beta_{r'}},
    \end{equation*}
    then
    \begin{equation}
      \label{eq:upper-0}
      z(t, x) \leq 2 x \e^{\beta_{\abs{t}}} \abs{t}^{-1/2} \m{I}(T).
    \end{equation}
    
    To derive a matching lower bound on $z$, we only integrate over $s \leq -t^2$ 
    in \eqref{eq:Duhamelbis2}.
    Note that if 
    \begin{equation}
      \label{26apr2012}
      \hbox{$t\ll -1$ and $x \leq \abs{t}/2$,}
    \end{equation} 
    then $s \leq -t^2$ yields
    \begin{equation*}
      \frac{x}{\sqrt{t - s + 1}} \leq 1.
    \end{equation*}
    Moreover, the function $\Phi(\eta)/\eta$ is decreasing on $\R_+$ because $\Phi(0) = 0$ and $\Phi$ is concave.
    It follows that for all $\eta \in (0,1]$,
    \begin{equation*}
      \frac{\Phi(\eta)}{\eta} \geq \Phi(1)=\frac{1}{\sqrt{\pi}}\int_0^1 \e^{-|y|^2/4}dy\approx{0.52} > \frac{1}{2}.
    \end{equation*}
    Thus \eqref{eq:equiv}, together with another application of \eqref{26apr2006}, gives a lower bound
    \begin{align}
      z(t, x) &\geq 2^{-2} x \e^{\beta_{\abs{t}}} \abs{t}^{-1/2} \int_{-T}^{-t^2} \e^{-\beta_{\abs{s}}} \abs{s}^{-1/2} (t - s + 1)^{-1/2} \d s\nonumber\\
              &\geq 2^{-2} x \e^{\beta_{\abs{t}}} \abs{t}^{-1/2} \int_{t^2}^T \frac{\dn r}{r \exp \beta_r} = 2^{-2} x \e^{\beta_{\abs{t}}} \abs{t}^{-1/2} [\m{I}(T) - \m{I}(t^2)].\label{eq:lower-0}
    \end{align}
    We now use the key hypothesis \eqref{eq:integral-condition}, which states that $\m{I}(\infty) = \infty$.
    We can therefore fix any $\f{t} \leq \tmax$ and choose $T = T(\f{t})$ so that
    \begin{equation}
      \label{eq:T-choice}
      \e^{\beta_{\abs{\f{t}}}} \abs{\f{t}}^{-1/2} [\m{I}(T) - \m{I}(\f{t}^2)] = 4.
    \end{equation}
    By \eqref{eq:T-choice} and \eqref{eq:beta-bounds}, we have
    \begin{equation*}
      \m{I}(T) \gtrsim \abs{\f{t}}^{1/4}.
    \end{equation*}
    Since $\abs{\beta_r} \ll \log r$, $\m{I}(r) \leq C(\delta) r^\delta$ for all $\delta > 0$.
    It follows that
    \begin{equation*}
      \m{I}(\f{t}^2) \ll \abs{\f{t}}^{1/4}.
    \end{equation*}
    Provided $\tmax$ is sufficiently negative, we can therefore assume that 
    \begin{equation*}
      \m{I}(\f{t}^2) \leq \m{I}(T)/2,
    \end{equation*}
    and hence \eqref{eq:T-choice} yields
    \begin{equation}
      \label{eq:T-bound}
      \e^{\beta_{\abs{\f{t}}}} \abs{\f{t}}^{-1/2} \m{I}(T) \leq 8.
    \end{equation}
    Combining \eqref{eq:T-bound} and \eqref{eq:upper-0}, and using \eqref{26apr2608}, we find
    \begin{equation}
      \label{eq:upper}
      z(t, x) \leq 16 x \ForAll t \leq \f{t}.
    \end{equation}
    Similarly, \eqref{eq:lower-0} and \eqref{eq:T-choice} yield
    \begin{equation}
      \label{eq:lower}
      z(\f{t}, x) \geq x \ForAll x \leq \frac{\abs{\f{t}}}{2}.
    \end{equation}
    We recall that the restriction on $x$ in \eqref{eq:lower} follows from \eqref{26apr2012}.
    This lower bound cannot hold for \emph{all} $x$ because it uses $\Phi(\eta) \geq \eta/2$, which fails for large $\eta.$
    This completes the proof of \eqref{eq:lowerbis}.

    It remains to show the derivative bound \eqref{eq:z-derivbis}.
    Combining \eqref{eq:p-small}, \eqref{eq:original-to-ss}, and \eqref{26apr1614}, we have $\abss{\partial_\eta \zeta} \leq 2$ provided $\tmax\ll -1$.
    This translates to
    \begin{equation*}
      \abs{\partial_x \zeta} \leq 2 (t - s + 1)^{-1/2}.
    \end{equation*}
    Using this bound and \eqref{26apr1620} in \eqref{eq:Duhamelbis2}, we have
    \begin{align*}
      |\partial_xz(t, x)| &\le 
                            2 \abs{t}^{-1/2} \e^{\beta_{\abs{t}}}\int_{-T}^t \abs{s}^{1/2} \e^{-\beta_{\abs{s}}} \left[\frac{1}{2 \abs{s}} - \dot{\beta}_{\abs{s}}\right] \frac{1}{(t - s + 1)^{1/2}} \d s\\
                          &\le 2 \abs{t}^{-1/2} \e^{\beta_{\abs{t}}}\int_{-T}^t \abs{s}^{-1/2} \e^{-\beta_{\abs{s}}} 
                            \frac{1}{(t - s + 1)^{1/2}} \d s.
    \end{align*}
    The right side above is exactly as in \eqref{26apr1621} except for the factor of $x$.
    Mimicking the arguments leading to \eqref{eq:upper}, we now obtain
    \begin{equation}
      \label{eq:z-deriv}
      \abs{\partial_x z(t,x)} \leq 16 \ForAll t \leq \f{t}.
    \end{equation}
    This finishes the proof of Lemma~\ref{lem-26apr2002}.
  \end{proof}
  We now resume the proof of Lemma~\ref{lem:der-mart-lower}.
  We apply Lemma~\ref{lem-26apr2002} with $\kappa$ given by \eqref{26apr1602}.
  We henceforth assume $t \leq \f{t}$.
  We assemble our subsolution of the form \eqref{26apr2016}, except we replace $z$ by $2^{-4}z$.
  The small factor $2^{-4}$ helps us achieve $\partial_x \ubar{u} \leq 0$.
  Indeed, recalling the exponential factor $\e^{-x}$ in \eqref{eq:1D-transform}, \eqref{eq:z-deriv} yields
  \begin{equation}
    \label{eq:decreasing}
    \partial_x[\e^{-x}(2^{-4} z + 1)] \leq -\e^{-x}(1 - 2^{-4} \partial_x z) \leq 0.
  \end{equation}
  Recall that $\cubar \leq 1$.
  By \eqref{eq:upper},
  \begin{equation}
    \label{26apr2018}
    \cubar \e^{-x}(2^{-4}z + 1) \leq \e^{-x}(x + 1) \leq 2.
  \end{equation}
  Let $a_z(t)$ solve 
  \begin{equation*}
    \dot a_z = - 2 \abs{t}^{-2} a_z^2, \quad a_z(-\infty) = 1.
  \end{equation*}
  A solution exists for all $t \ll -1$ because $\abs{t}^{-2}$ is integrable at infinity.
  Then a direct computation shows that
  \begin{equation}
    \label{26apr2331}
    \ti w \coloneqq \cubar a_z(t)(2^{-4}z + 1) 
  \end{equation}
  is a subsolution of the nonlinear equation in \eqref{eq:shift}: by \eqref{eq:drift-signs}, \eqref{eq:finite-time}, and \eqref{26apr2018},
  \begin{align*}
    \partial_t \ti w&- \partial_x^2 \ti w - \left[\frac{1}{2\abs{t}} - \dot{\beta}_{\abs{t}}\right] \ti w +\kappa(t) \partial_x \ti w + \abs{t}^{-2} \e^{-x} \ti w^2\\
                    &= \cubar\dot a_z(2^{-4}z + 1) - \cubar a_z \left[\frac{1}{2\abs{t}} - \dot{\beta}_{\abs{t}}\right] [1 - 2^{-4} \Phi(x)] + |t|^{-2}\e^{-x}\cubar^2a_z^2(2^{-4}z +1)^2\\
                    &\leq -\cubar \abs{t}^{-2} a_z^2 (2^{-4}z + 1)[2 - \cubar \e^{-x}(2^{-4}z +1)] \leq 0
  \end{align*}
  We also have the boundary comparison for \eqref{eq:shift}: $\ti w(t, 0) \leq \cubar$.
  Since $a_z(-\infty) = 1$, we may assume $t_{\text{max}}$ is sufficiently negative that
  \begin{equation}
    \label{26apr2335}
    a_z(t) \geq 1/2 \ForAll t\le t_{\text{max}}.
  \end{equation}
  Then \eqref{eq:lower} yields
  \begin{equation}
    \label{eq:w-lower}
    \ti{w}(\f{t}, x) \geq \frac{\cubar}{2}[2^{-4}z(\f{t},x) + 1] \geq \cubar 2^{-5}x + \frac{\cubar}{2} \ForAll x \leq \abs{\f{t}}/2.
  \end{equation}
  
  We next claim that
  \begin{equation}
    \label{eq:log-deriv}
    \abs{\partial_x z} \leq \frac{z}{8} \ForAll x \geq 2^9.
  \end{equation}
  Since $z$ is a positive superposition \eqref{eq:Duhamelbis2} of $\zeta(t,x;s)$ with coefficients that do not depend on $x$, it suffices to show the same for $\zeta(t, x; s)$.
  By \eqref{eq:original-to-ss}, \eqref{26apr1614} and \eqref{eq:p-small}, we have 
  \begin{equation}
    \label{eq:zeta-deriv}
    \abs{\partial_x \zeta} = (t - s + 1)^{-1/2} \absb{\partial_\eta(\Phi - \e^{-\eta^2/8}p)} \leq 2 (t - s + 1)^{-1/2} \e^{-\eta^2/16}.
  \end{equation}
  Also, \eqref{eq:zeta-error} yields 
  \begin{equation}
    \label{26apr2329}
    \zeta \geq \Phi(\eta)/2.
  \end{equation}
  If $\eta \leq 2^4$, then one can numerically verify $\Phi(\eta) \geq 2^{-5}\eta$, and \eqref{26apr2329} gives 
  \begin{equation*}
    \zeta \geq 2^{-5} x (t - s + 1)^{-1/2}.
  \end{equation*}
  Thus $x \geq 2^9$ and \eqref{eq:zeta-deriv} imply that
  \begin{equation*}
    \abs{\partial_x \zeta}\le 2 (t - s + 1)^{-1/2}\le 2^6\frac{\zeta}{x}\le\frac{\zeta}{8}.
  \end{equation*} 
  On the other hand, a numerical calculation yields
  \begin{equation*}
    2^5 \e^{-\eta^2/16} \leq \Phi(\eta) \ForAll \eta\geq 16.
  \end{equation*}
  Using $t-s+1\ge 1$ in \eqref{eq:zeta-deriv}, \eqref{26apr2329} yields
  \begin{equation*}
    \abss{\partial_\eta \zeta} \leq 2 \e^{-\eta^2/16} \leq \frac{\Phi(\eta)}{16} \leq \frac{\zeta}8.
  \end{equation*}
  We conclude that in both cases 
  \begin{equation*}
    \abss{\partial_x \zeta} \leq \frac{\zeta}8 \ForAll x \geq 2^8.
  \end{equation*}
  Taking a weighted integral over $s$ in \eqref{eq:Duhamelbis2}, we see that  the same holds for $z$, confirming \eqref{eq:log-deriv}.

  We use \eqref{eq:log-deriv} to help cut off $z$ where $x \geq \abs{t}/2$.
  Given $M > 0$, define
  \begin{equation}
    \label{26apr2332}
    a_m(t) \coloneqq \exp(-M\e^{t/8}) \And m(t,x) \coloneqq a_m(t)\big[1 - \e^{(2x + t)/8}\big].
  \end{equation}
  We claim that we can choose $M$ sufficiently large that $m\ti w$ is a subsolution of the nonlinear equation~\eqref{eq:shift} in the region where $x \leq \abs{t}/2$.
  The zeroth-order linear term is unaffected by a multiplier, and
  \begin{equation*}
    -m\ti w^2 \leq -(m\ti w)^2
  \end{equation*}
  because $m \leq 1$, so the nonlinear term in~\eqref{eq:shift} is helpful.
  Thus, it suffices to show that
  \begin{equation*}
    \m{Q} \coloneqq (\partial_t - \partial_x^2 + \kappa \partial_x)(m\ti w) - m(\partial_t - \partial_x^2 + \kappa \partial_x)\ti w \leq 0.
  \end{equation*}
  Computing, we get 
  \begin{align}
    \m{Q} &= \ti w (\partial_t - \partial_x^2 + \kappa \partial_x) m - 2 \partial_xm \partial_x \ti w\nonumber\\
          &= -\frac{M}{8} \e^{t/8} m \ti w - \frac{a_m(t)}{2}\e^{(2x + t)/8}\Big(\frac{1 + 4 \kappa}{8}\ti w - \partial_x \ti w\Big).\label{26apr2334}
  \end{align}
  Recall from \eqref{eq:drift-signs} that $\kappa > 0$.
  Thus, if we know that 
  \begin{equation*}
    \partial_x \ti w \leq \frac{\ti w}{8},
  \end{equation*}
  then $\m{Q} \leq 0$.
  If $x \geq 2^9$, this follows from \eqref{eq:log-deriv} and the definition \eqref{26apr2331} of $\ti w$.

  Now suppose $x \leq 2^9$.
  By \eqref{eq:z-deriv}, \eqref{26apr2331},  and \eqref{26apr2332}, the derivative term in \eqref{26apr2334} can be bounded from above as
  \begin{equation*}
    a_m \e^{(2x + t)/8} \partial_x \ti w \leq \cubar \e^{2^7} a_m \e^{t/8}.
  \end{equation*}
  On the other hand, because $x \leq 2^9$, we may assume that $t$ is sufficiently negative that $\e^{(2x+t)/8} \leq 1/2$.
  Then $m \geq a_m/2$ and \eqref{eq:w-lower} allows us to bound the first term in the right side of \eqref{26apr2334} by
  \begin{equation*}
    \frac{M}{8} \e^{t/8} m \ti w \geq \cubar 2^{-5} M a_m \e^{t/8}.
  \end{equation*}
  Thus if we choose $M = 2^5 \e^{2^7}$, then $\m{Q} \leq 0$ and $m \ti w$ is a subsolution.
  Since $a_m(t) \to 1$ as $t \to -\infty$, we may assume that $t$ is sufficiently negative that $a_m(t) \geq 1/2$.
  This complements \eqref{26apr2335}.
  We can finally define $\ubar{T} \coloneqq T$ and
  \begin{equation*}
    \ubar{w} \coloneqq m \ti w, \quad \ubar{v}(t, x; \f{t}) \coloneqq \ubar{w}(t,x - 2 \log\abs{t}; \f{t}), \And \ubar{u}(t, x; \f{t}) \coloneqq \e^{-x} \ubar v(t, x; \f{t}).
  \end{equation*}
  We highlight the dependence on $\f{t}$ through the starting time $-T(\f{t})$ in Lemma~\ref{lem-26apr2002}, which was chosen to satisfy \eqref{eq:T-choice}.

  We now confirm that $\ubar{u}$ has the desired properties.
  The cutoff $m$ ensures that 
  \begin{equation*}
    \hbox{$\ubar{u}(t,x) = 0$ where $x \geq \dfrac{\abs{t}}2 + 2 \log \abs{t}$},
  \end{equation*}
  which includes the region $x \geq \abs{t}$ provided $\tmax \ll -1$.
  Where $x = 2 \log \abs{t}$, we have $\ubar{u} \leq \cubar \e^{-x}$.
  The same holds at $t = -\ubar{T}$ because $z = 0$ there.
  By \eqref{eq:decreasing}, $\partial_x \ubar{u} \leq 0$.
  By construction, $\ubar{u}$ is a subsolution of \eqref{eq:ancient}.
  
  Finally, where $x \leq \abs{t}/4$, we may assume $m \geq 1/4$.
  Then \eqref{eq:w-lower} yields
  \begin{align*}
    \ubar{u}(\f{t}, x; \f{t}) = \e^{-x} \ubar{w}(\f{t}, x - 2 \log \abs{\f{t}}; \f{t}) &\geq 2^{-2} \e^{-x} \ti w(\f{t}, x - 2 \log \abs{\f{t}}; \f{t})\\
                                                                                       &\geq 2^{-7} \cubar \e^{-x} (x - 2 \log \abs{\f{t}}).
  \end{align*}
  If $x \geq 4 \log \abs{\f{t}}$, we obtain
  \begin{equation*}
    \ubar{u}(\f{t}, x; \f{t}) \geq 2^{-8} \cubar x \e^{-x}.
  \end{equation*}
  This finishes the proof of Lemma~\ref{lem:der-mart-lower}.
  As we have noted below its statement, this also completes the proof of Proposition~\ref{prop:der-mart-multi-bis2}.
\end{proof}

\section{The proof of Proposition~\ref{prop:der-mart-multi-bis}}
\label{sec:excursions-proof}

Fix $T\ge 2$ and let $h_T$ be the solution to~\eqref{26apr2312}:
\begin{equation}
  \label{26apr2402}
  \begin{cases}
    \partial_th_T=\Delta h_T+h_T-h_T^2=0, & t\in\R, x\in\m{A}^\cc(-t),\\
    h_T(t,x)=1, & t<-T,~x\in\m{A}(-t),
  \end{cases}
\end{equation}
with the set $\m{A}(t)$ defined in \eqref{26apr2404}--\eqref{26apr2306}.
We need to show that for every nonempty open $U' \Subset U$, there exists $\tmax(U') \leq -T$ and $c(d,U') > 0$ such that
\begin{equation}
  \label{26apr2408}
  h_T(t, r, \theta) \geq c [b_{\abs{t}} - r] \e^{-[b_{\abs{t}} - r]}
\end{equation}
for all $t \leq \tmax$, $\theta \in U'$, and $r$ satisfying $4 \log \abs{t} \leq b_{\abs{t}} - r \leq \abs{t}/8$.

In spherical coordinates, \eqref{26apr2402} becomes
\begin{equation}
  \label{eq:hitting}
  \begin{cases}
    \partial_t h_T = \partial_r^2 h_T + 
    \frac{d-1}{r} \partial_r h_T + \frac{1}{r^2} \Delta_\theta h_T + h_T - h_T^2, & t\in\R, x\in\m{A}^\cc(-t),\\
    h_T(t,x) = 1, & t<-T,~x\in\m{A}(-t).
  \end{cases}
\end{equation}
We transform the one-dimensional subsolutions $\mr{u}$ and $\ubar{u}$ of \eqref{eq:ancient} 
into lower bounds on $h_T.$
Before we define them formally, let us first explain the intuition.
Let $x$ in the one-dimensional problem \eqref{eq:ancient} represent the distance $b_{\abs{t}} - r$ 
inside the threshold in \eqref{26apr2408}, 
\begin{equation*}
  r = b_{\abs{t}} - x.
\end{equation*}
In the $b_{\abs{t}}$-moving frame, the drift in \eqref{eq:hitting} is
\begin{equation*}
  \dot{b}_{\abs{t}} + \frac{d-1}{r}.
\end{equation*}
In comparison, the drift in \eqref{eq:ancient} is $\dot{b}_{\abs{t}}^{\textnormal{1D}}$.
Using \eqref{eq:threshold} and \eqref{eq:1D-barrier}, their difference is
\begin{equation}
  \label{eq:drift-diff}
  \dot{b}_{|t|} + \frac{d-1}{r} - \dot{b}_{|t|}^{\textnormal{1D}} = \frac{d-1}{r} - \frac{d-1}{2\abs{t}}.
\end{equation}
This is positive whenever $r < 2\abs{t}$.
Because $\mr{u}$ and $\ubar{u}$ are decreasing in $x$, the difference \eqref{eq:drift-diff} in drifts favors a subsolution in this region.
Thus we only need to treat an unfavorable discrepancy in the $\m{O}(\log \abs{t})$ region where $r > 2 \abs{t}$, and there the difference in \eqref{eq:drift-diff} is quite small.
To do so, let $\xi \colon (-\infty, -1) \to \R$ solve
\begin{equation}
  \label{eq:mismatch}
  \dot \xi = (d-1) \Big(\frac{1}{2\abs{t}} - \frac{1}{b_{\abs{t}}}\Big)_+, \quad \xi(-\infty) = 0.
\end{equation}
The solution to \eqref{eq:mismatch} is bounded because the right side is of order $|t|^{-2}\log \abs{t}$.
Shifting $u$ by the slowly moving but bounded quantity $\xi$ will correct the slight mismatch in first-order terms between \eqref{eq:ancient} and \eqref{eq:hitting}.

Next, we account for the angular dependence of the boundary condition in \eqref{eq:hitting}.
Let $V$ be a smooth open set such that $U' \Subset V \subset U$.
Let $\lambda, \psi > 0$ denote the principal eigenvalue and eigenfunction of the Dirichlet problem
\begin{equation}
  \label{26apr2414}
  \begin{cases}
    -\Delta_\theta\psi = \lambda\psi & \text{in }V,\\
    \psi = 0 & \text{on } \partial V,
  \end{cases}
\end{equation}
normalized so that $\sup_V \psi = 1$.
When we apply $\Delta_\theta$ to $\psi$, we pick up a factor of $\lambda$, which is 
accompanied by $r^{-2}$ in \eqref{eq:hitting}.
Because $\mr{u}$ and $\ubar{u}$ vanish where $x \geq \abs{t}$, or equivalently 
$r \leq b_{\abs{t}} - \abs{t}$, we can bound $r^{-2}$ from above in \eqref{eq:hitting} by 
\begin{equation*}
  [b_{\abs{t}} - \abs{t}]^{-2} \lesssim \abs{t}^{-2}.
\end{equation*}
This term is integrable at infinity, so we can absorb the angular term through a slowly shrinking multiplier.
To this end, define $a_\theta$ satisfying
\begin{equation}
  \label{26apr2416}
  \dot{a}_\theta = -\frac{\lambda}{[b_{\abs{t}} - \abs{t}]^2} a_\theta, \quad a_\theta(-\infty) = 1.
\end{equation}
Then $1 \lesssim a_\theta \leq 1$.

We can now assemble the pieces discussed above.
As in the proof of the one-dimensional Proposition~\ref{prop:der-mart-multi-bis2}, we proceed in two steps.
First, we prove a purely exponential lower bound, as in Lemma~\ref{lem:exp-lower}, and then improve it to include the linear factor $b_{|t|}-r$, as in Lemma~\ref{lem:der-mart-lower}.

Take $\tmax \leq -1$ from Lemma~\ref{lem:der-mart-lower} and fix $\f{t} \leq \tmax$.
Define
\begin{equation*}
  \mr{h}(t,r,\theta; \f{t}) \coloneqq a_\theta(t) \psi(\theta) \mr{u}(t, b_{\abs{t}} + \xi(t) - r; \f{t})
\end{equation*}
for $t \leq \tmax$, $\theta \in \bar{V}$, and 
\begin{equation}
  \label{26apr2702}
  b_{\abs{t}} - r \in [0, \abs{t} - \xi(t)].
\end{equation}
Here $\mr{u}$ is the one-dimensional subsolution constructed in Lemma~\ref{lem:exp-lower}, and the functions $a_\theta$, $\psi$, and $\xi$ were constructed above.
We claim $\mr{h}$ is a subsolution for $h_T$.

The comparison $\mr{h}\le h_T$ holds at the boundary portion $\{r=b_{|t|},\theta\in U,t<-T\}$ because there
\begin{equation*}
  \mr{h} \leq \psi(\theta) \leq \tbf{1}_U(\theta) = h_T.
\end{equation*}
Likewise, where $\theta \in \partial V$, $\mr{h} = 0 \leq h_T$.
This accounts for the boundary data.
For the initial data, $\mr{u}$ and thus $\mr{h}$ vanishes in the distant past, 
and is then less than $h_T$.

Therefore, we only need to verify that $\mr{h}$ satisfies the appropriate differential inequality.
Using Lemma~\ref{lem:exp-lower}~\ref{item:exp-sub} as well as \eqref{26apr2416}, we compute, for $t<0$:
\begin{align}
  \partial_t \mr{h} &= -\frac{\lambda}{[b_{\abs{t}} - \abs{t}]^2} \mr{h} 
                      + a_\theta \psi \partial_t \mr{u} - a_\theta \psi [\dot{b}_{\abs{t}} - \dot \xi(t)] \partial_x \mr{u}\nonumber\\
                    &\le-\frac{\lambda}{[b_{\abs{t}} - \abs{t}]^2} \mr{h}+a_\theta \psi
                      \big(\partial_x^2 \mr{u} + \dot{b}_{\abs{t}}^{\textnormal{1D}} \partial_x \mr{u} + \mr{u} - 
                      \mr{u}^2 -[\dot{b}_{\abs{t}} - \dot \xi(t)] \partial_x \mr{u}\big)\nonumber\\
                    &\leq -\frac{\lambda}{[b_{\abs{t}} - \abs{t}]^2} \mr{h} + \partial_r^2 \mr{h} 
                      + [\dot{b}_{\abs{t}} - \dot \xi(t) -\dot{b}_{\abs{t}}^{\textnormal{1D}}] \partial_r \mr{h} + \mr{h} - (a_\theta \psi)^{-1}\mr{h}^2.\label{26apr2708}
\end{align}
We now estimate the terms in the right side of \eqref{26apr2708}.
For the drift, \eqref{eq:mismatch} and the definition \eqref{eq:1D-barrier} of ${b}_{t}^{\textnormal{1D}}$ yield
\begin{align*}
  \dot{b}_{\abs{t}} - \dot \xi(t)-\dot{b}_{\abs{t}}^{\textnormal{1D}}&=
                                                                       2+ \frac{d-2}{2 \abs{t}} + \dot{\beta}_{\abs{t}} - 
                                                                       \dot{\xi}(t) -2+\frac{1}{2|t|}-\dot\beta_{|t|}\\
                                                                     &=\frac{d-1}{2 \abs{t}}-(d-1) \Big(\frac{1}{2\abs{t}} - \frac{1}{b_{\abs{t}}}\Big)_+
                                                                       \leq \frac{d-1}{b_{\abs{t}}} \leq \frac{d-1}{r},
\end{align*}
because $b_{|t|}\ge r$ by \eqref{26apr2702}.
By Lemma~\ref{lem:exp-lower}, $\partial_x \mr{u} \leq 0$, which implies that $\partial_r \mr{h} \geq 0$.
Therefore, the middle term in the right side of \eqref{26apr2708} can be bounded as 
\begin{equation}
  \label{26apr2710}
  [\dot{b}_{\abs{t}} - \dot \xi(t) -\dot{b}_{\abs{t}}^{\textnormal{1D}}] \partial_r \mr{h} \leq \frac{d-1}{r} \partial_r \mr{h}.
\end{equation}
The eigenvalue problem \eqref{26apr2414} implies that
\begin{equation}
  \label{26apr2706}
  \Delta_\theta \mr{h} = -\lambda \mr{h}.
\end{equation}
In addition, \eqref{26apr2702} and positivity of $\xi(t)$ given by \eqref{eq:mismatch} imply that
\begin{equation*}
  0<b_{|t|}-|t|\le r-\xi(t)\le r.
\end{equation*}
Together with \eqref{26apr2706}, this bounds the first term in the right side of \eqref{26apr2708} as  
\begin{equation}
  \label{26apr2712}
  -\frac{\lambda}{[b_{\abs{t}} - \abs{t}]^2} \mr{h} 
  \leq -\frac{\lambda}{r^2} \mr{h} = \frac{1}{r^2} \Delta_\theta \mr{h}.
\end{equation}
Finally, as $a_\theta,\psi \leq 1$, the last term in \eqref{26apr2708} satisfies
\begin{equation}
  \label{26apr2714}
  -(a_\theta \psi)^{-1}\mr{h}^2 \leq -\mr{h}^2.
\end{equation}
Inserting \eqref{26apr2710}, \eqref{26apr2712} and \eqref{26apr2714} 
into \eqref{26apr2708}, we obtain 
\begin{equation*}
  \partial_t \mr{h} \leq \partial_r^2 \mr{h} + \frac{d-1}{r} \partial_r \mr{h} 
  + \frac{1}{r^2} \Delta_\theta \mr{h} + \mr{h} - \mr{h}^2.
\end{equation*}
Thus, $\mr{h}$ is indeed a subsolution on its support.

By the comparison principle, we deduce that $h_T\geq \mr{h}$.
Using \eqref{26apr1320_1} from Lemma~\ref{lem:exp-lower}, we find
\begin{equation*}
  h_T(\f{t}, r, \theta) \geq \mr{h}(\f{t},r,\theta; \f{t}) \geq 
  a_\theta(\f{t}) \psi(\theta) \cexp \e^{-[b_{\abs{\f t}} + \xi(\f t) - r]} 
  \ForAll b_{\abs{\f{t}}} - r \in [0, \abs{\f{t}}/2].
\end{equation*}
Let $\cubar \coloneqq \cexp \e^{-\xi(\tmax)}$.
Because $\xi$ is increasing and $\f{t} \leq \tmax$ was arbitrary, we have
\begin{equation}
  \label{eq:h-exp}
  h_T(t, r, \theta) \geq  a_\theta(t) \psi(\theta) \cubar \e^{-[b_{\abs{t}} - r]} 
  \ForAll t \leq \tmax, \enspace b_{\abs{t}} - r \in [0, \abs{t}/2].
\end{equation}
This exponential lower bound on $h_T$ is the analog of Lemma~\ref{lem:exp-lower}.

Next, we use lower bound \eqref{eq:h-exp} at the location $r=b_{|t|}-2\log|t|$ to bootstrap an additional linear factor via Lemma~\ref{lem:der-mart-lower}.
We define
\begin{equation*}
  \ubar{h}(t,r,\theta; \f{t}) \coloneqq a_\theta(t) \psi(\theta) \ubar{u}(t, b_{\abs{t}} + \xi(t) - r; \f{t})
\end{equation*}
for $t \leq \f{t} \leq \tmax$, $\theta \in \bar{V}$, and 
\begin{equation*}
  b_{\abs{t}} - r \in [2 \log \abs{t}, \abs{t} - \xi(t)].
\end{equation*}
Here, $\ubar{u}$ is the subsolution constructed in Lemma~\ref{lem:der-mart-lower}, with $\cubar = \cexp \e^{-\xi(\tmax)}$ as above.
The same argument used for $\mr{h}$ shows that $\ubar{h}$ is a subsolution of the evolution equation in \eqref{eq:hitting}.

For the boundary data, $\ubar{h} = 0 \leq h_T$ where $\theta \in \partial V$.
Where $r = b_{\abs{t}} - 2 \log \abs{t}$, Lemma~\ref{lem:der-mart-lower}~\ref{item:der-mart-upper}, $\xi \geq 0$, and \eqref{eq:h-exp} ensure that
\begin{equation*}
  \ubar{h} \leq a_\theta \psi \cubar \e^{-[b_{\abs{t}} + \xi(t) - r]} 
  \leq a_\theta \psi \cubar \e^{-[b_{\abs{t}} - r]} \leq h_T.
\end{equation*}
The same relation holds at time $t = -\ubar{T} < \f{t}$ where $x \leq \abs{\ubar{T}}/2$, so $\ubar{h} \leq h_T$ both at an initial time and along appropriate boundaries.
By the comparison principle, we conclude that $h_T \geq \ubar{h}$.
Using \eqref{26apr1320} from Lemma~\ref{lem:der-mart-lower}, we obtain
\begin{align*}
  h_T(\f{t},r,\theta) \geq \ubar{h}(\f{t},r,\theta; \f{t}) &\geq a_\theta(\f t) \psi(\theta) \ubar{u}(\f{t},b_{\abs{\f{t}}} + \xi(\f{t}) - r; \f{t})\\
                                                           &\geq a_\theta(\f t) \psi(\theta) \e^{-\xi(\tmax)} 2^{-8} \cubar [b_{\abs{\f{t}}} - r] \e^{-[b_{\abs{\f{t}}} - r]}
\end{align*}
for all 
\begin{equation*}
  b_{\abs{\f{t}}} - r \in [4 \log\abs{\f{t}},\abs{\f{t}}/4 - \xi(\f{t})].
\end{equation*}
In particular, this holds for 
\begin{equation*}
  b_{\abs{\f{t}}} - r \in [4 \log\abs{\f{t}},\abs{\f{t}}/8].
\end{equation*}
As $U' \Subset V$, we know  that $\inf_{U'} \psi > 0$.
Taking $\tmax \ll -1$, we can arrange $a_\theta \geq 1/2$ and $\e^{-\xi(\tmax)} \geq 1/2$.
Setting 
\begin{equation*}
  c(d,U') \coloneqq 2^{-10} \cubar \inf_{U'} \psi \geq 2^{-11} \cexp \inf_{U'} \psi > 0,
\end{equation*}
we find
\begin{equation*}
  h(t,r,\theta) \geq c [b_{\abs{t}} - r] \e^{-[b_{\abs{t}} - r]}
\end{equation*}
for all $t \leq \tmax$, $\theta \in U'$, and $b_{\abs{t}} - r \in [4 \log\abs{t},\abs{t}/8]$.
This finishes the proof of Proposition~\ref{prop:der-mart-multi-bis}.

\section{Necessity of integral condition}
\label{sec:necessary}

As a complement to the work above, we now show that if the integral in \eqref{eq:integral-condition} is finite, then the BBM will eventually lie inside the radius $b_t$.
We therefore assume that \eqref{eq:beta-dot} holds and
\begin{equation}
  \label{eq:necessary}
  \int_1^\infty \frac{\dn r}{r \exp \beta_r} < \infty.
\end{equation}
Let $E_T(S^{d-1})$ be the event that a particle of the BBM crosses the radial threshold~$r=b_t$ after time $T$.
The main result of this section is the following.
\begin{proposition}
  \label{prop:necessary}
  If $\beta$ satisfies \eqref{eq:beta-dot} and \eqref{eq:necessary}, then
  \begin{equation*}
    \lim_{T \to \infty} \P[E_T(S^{d-1})] = 0.
  \end{equation*}
\end{proposition}
Recall from Lemma~\ref{lem:ancientbis} that the probability
\begin{equation*}
  \bar h_T(t,x):=\P_{t,x}[E_T(S^{d-1})]
\end{equation*}
is the minimal positive solution to 
the global-in-time backward problem \eqref{eq:ancientbis}:
\begin{equation*}
  \begin{cases}
    \partial_t\bar h_T+\Delta \bar h_T+\bar h_T-\bar h_T^2=0, & t\in\R, x\in\m{A}^\cc(t)\\
    \bar h_T(t,x)=1, & t>T,~x\in\m{A}(t).
  \end{cases}
\end{equation*}
In the present case, the set $\m{A}(t)$ is defined simply as 
\begin{equation*}
  \m{A}(t)=\{(R,\theta)\in\R^d:~R=b_t\},
\end{equation*}
without any restriction on the angle, and 
with the same convention \eqref{26apr2306}:
\begin{equation*}
  \m{A}(t)=\emptyset \For t<T.
\end{equation*}
Flipping the direction of time and using the radial symmetry of $\bar h_T(t,x)$, we can set
\begin{equation*}
  h_T(t,r)=\bar h_T(-t,x), \quad |x|=r,
\end{equation*}
and arrive at \eqref{eq:hitting} 
\begin{equation*}
  \begin{cases}
    \partial_t h_T = \partial_r^2 h_T + \frac{d-1}{r} \partial_r h_T   + h_T - h_T^2, & t>-T,~r>0 \text{ or } t \leq -T, r \neq b_{-t},\\
    h_T(t,r) = 1, & t<-T,~r=b_{-t}.
  \end{cases}
\end{equation*}
To prove Proposition~\ref{prop:necessary}, we need to show that  
\begin{equation}
  \label{26may414}
  \lim_{T\to+\infty}h_T(0,0)=0.
\end{equation}
We first show that this follows from a slightly different result for a modified threshold.

If $\beta$ satisfies \eqref{eq:beta-dot} and \eqref{eq:necessary}, one can readily check that there exists $\ti\beta_s$ satisfying the same conditions such that
\begin{equation*}
  \beta_s - \ti \beta_s \to \infty \As s \to \infty.
\end{equation*}
We denote the corresponding radius by 
\begin{equation*}
  \tilde b_s=2s+\frac{d-2}{2}\log s+\tilde\beta_s,
\end{equation*}
so
\begin{equation}
  \label{eq:divergence}
  b_s - \ti b_s \to \infty \As s \to \infty.
\end{equation}
We study the auxiliary function $\tilde h$ that is the minimal ancient solution of 
\begin{equation}
  \label{26may412}
  \begin{cases}
    \partial_t \tilde h = \partial_r^2 \tilde h + 
    \frac{d-1}{r} \partial_r \tilde h + \tilde h - \tilde h^2, & t\leq -1, r < \tilde{b}_{-t},\\
    \tilde h(t,r) = 1, &r=\tilde{b}_{-t}.
  \end{cases}
\end{equation}
Note that \eqref{26may412} does not depend on $T$, but the threshold $\tilde b_t$ moves slower than $b_t$.
The function $\tilde h(t,x)$ is the probability that a BBM started at time $\abs{t}$ and position $x\in\R^d$  hits the threshold $s \mapsto \ti b_s$ at some time after $\abs{t}$.

We claim that in order to prove \eqref{26may414}, it suffices to show that
\begin{equation}
  \label{26may416}
  \hbox{$\tilde h(t,0) \to 0$ as $t \to -\infty$.}
\end{equation}
Indeed, assume that \eqref{26may416} holds, fix $\eps > 0$, and take $\tau_\eps\gg 1$ such that 
\begin{equation}
  \label{26may422}
  \tilde h(-\tau_\eps, 0) \leq \eps.
\end{equation}
Consider the time-shifted curve 
\begin{equation*}
  \h b_s^{(\eps)} \coloneqq \ti b_{\tau_\eps + s}.
\end{equation*}
Note that
\eqref{eq:beta-dot} implies in particular that $\ti b$ is uniformly Lipschitz, so
\begin{equation*}
  |\h b^{(\eps)}_s-\ti b_s| = \abss{\ti b_{s + \tau_\eps} - \ti b_s}\leq C(\eps) \ForAll s \geq 1.
\end{equation*}
By \eqref{eq:divergence}, it follows that there exists $T_\eps \gg 1$ such that 
\begin{equation*}
  b_s \geq \h b_s^{(\eps)} \ForAll s \geq T_\eps.
\end{equation*}
Therefore, the probability $h_{T_\eps}(0,0)$ that the BBM started from $t = 0$ and $x = 0$ hits~$b_s$ after time $T_\eps$ 
is no larger than the probability that it hits $\hat b_s^{(\eps)}$ at all.
Shifting in time, the latter equals the probability $\tilde h(-\tau_\eps, 0)$ that the BBM started from $t=\tau_\eps$ and $x=0$ hits the curve $\tilde b_s$.
Using \eqref{26may422}, we deduce that
\begin{equation*}
  h_{T_\eps}(0,0)\le \tilde h(-\tau_\eps, 0)\le \eps.
\end{equation*}
This proves \eqref{26may414}.

Thus Proposition~\ref{prop:necessary} follows from \eqref{26may416}, and this is what we prove in the rest of this section.
We will drop the tilde in the notation for $\tilde h$, $\tilde b_t$ and $\tilde\beta_t$, writing \eqref{26may412} as simply
\begin{equation}
  \label{26may423}
  \begin{cases}
    \partial_t h = \partial_r^2 h + 
    \frac{d-1}{r} \partial_r  h + h - h^2, & t \leq -1, r < {b}_{-t},\\
    h(t,r) = 1, &r= {b}_{-t}.
  \end{cases}
\end{equation}

In order to prove  \eqref{26may416}, we will construct a useful supersolution for \eqref{26may423} by taking an isotropic superposition of 1D solutions.
To construct a supersolution, we are free to drop the negative nonlinearity in \eqref{26may423}, 
leaving a linear problem.
After a suitable exponential tilt, our building block is the minimal ancient solution of the following one-dimensional heat equation with a drift:
\begin{equation}
  \label{eq:upper-heat}
  \begin{cases}
    \partial_t v = \partial_y^2 v + \Big(\frac{d-2}{2 \abs{t}} 
    + \dot\beta_{\abs{t}}\Big) \partial_y v + \Big(\frac{1}{2 \abs{t}} - \dot\beta_{\abs{t}}\Big)v, & y > 0,\\
    v(t,0) = 1.
  \end{cases}
\end{equation}
This is very similar to \eqref{eq:shift} and \eqref{eq:z-trial}.
\begin{lemma}
  \label{lem:upper-heat}
  Suppose $\beta$ satisfies \eqref{eq:beta-dot} and \eqref{eq:necessary}.
  Then there exist constants $\tmax \leq -1$ and $C > 0$ such that the minimal positive ancient solution of \eqref{eq:upper-heat} satisfies
  \begin{equation}
    \label{26apr2804}
    v(t, y) \leq C(y + 1) \ForAll t \leq \tmax,\enspace y > 0.
  \end{equation}
\end{lemma}
\begin{proof}  
  We build an ancient supersolution of \eqref{eq:upper-heat} satisfying the desired bound.
  Recalling $\Phi$ from \eqref{eq:erf}, consider the minimal ancient solution of
  \begin{equation*}
    \begin{cases}
      \partial_t z = \partial_y^2 z - \kappa(t) \partial_y z + \left(\frac{1}{2 \abs{t}} - \dot\beta_{\abs{t}}\right)[z + \Phi(y)], & y > 0,\\
      z(t,0) = 0.
    \end{cases}
  \end{equation*}
  This equation is identical to \eqref{eq:finite-time}, but here we set
  \begin{equation}
    \label{26apr2802}
    \kappa(t) \coloneqq -\frac{d-2}{2 \abs{t}} - \dot\beta_{\abs{t}}.
  \end{equation}
  This differs from our earlier definition \eqref{26apr1602} of $\kappa$, but this has no effect on the analysis, as $\kappa$ still satisfies the hypothesis \eqref{26may516} of Lemma~\ref{lem-26apr2002}.

  The function $z(t,x)$ is given by \eqref{eq:Duhamelbis2} with $T = \infty$:
  \begin{equation*}
    z(t, y) = \lambda(t) \int_{-\infty}^t \abs{s}^{1/2} \e^{-\beta_{\abs{s}}} \left[\frac{1}{2 \abs{s}} - \dot{\beta}_{\abs{s}}\right] \zeta(t, y; s) \d s,
  \end{equation*}
  with $\lambda(t)$ as in \eqref{26apr1620}.
  The functions $\zeta(t, y; s)$ are, as before, solutions to \eqref{eq:Duhamel-piecebis}, with $\kappa$ now given by \eqref{26apr2802}.

  As in the proof of Lemma~\ref{lem-26apr2002}, by~\eqref{eq:zeta-error}, we have
  \begin{equation*}
    \zeta(t, y; s) \leq 2 \Phi\left(\frac{y}{t - s + 1}\right)
  \end{equation*}
  provided $t \leq \tmax \ll -1$.
  Since $\Phi(y) \leq y$, the analogue of \eqref{eq:upper-0} with \eqref{eq:necessary} is
  \begin{equation}
    \label{eq:lin-bd}
    z(t, y) \leq 2y\e^{\beta_{\abs{t}}} \abs{t}^{-1/2}\!\! \int_1^\infty \frac{\dn r}{r \exp \beta_r} \lesssim y\e^{\beta_{\abs{t}}} \abs{t}^{-1/2} \ForAll t \leq \tmax, y > 0.
  \end{equation}

  Since $\Phi$ is increasing and $\dot\beta_{|t|}\ll |t|^{-1}$, the function
  \begin{equation*}
    z_1(t,y)=\Phi(1)^{-1} z(t, y + 1)
  \end{equation*}
  satisfies
  \begin{equation*}
    \begin{cases}
      \partial_t z_1 \ge \partial_y^2 z_1 - \kappa(t) \partial_y z_1 + 
      \Big(\frac{1}{2 \abs{t}} - \dot\beta_{\abs{t}}\Big)(z_1 + 1), & y > 0,\\
      z_1(t,0) = 0.
    \end{cases}
  \end{equation*}
  Hence $z_1 + 1$ is a super-solution of the original problem \eqref{eq:upper-heat}.
  It follows that the minimal positive ancient solution of \eqref{eq:upper-heat} satisfies
  \begin{equation*}
    v(t, y) \leq z_1(t, y) + 1 \leq \Phi(1)^{-1} z(t, y + 1) + 1.
  \end{equation*}
  Then \eqref{26apr2804} follows from \eqref{eq:lin-bd}.
\end{proof}
We now use Lemma~\ref{lem:upper-heat} to prove \eqref{26may416}.
First note that if $v$ solves \eqref{eq:upper-heat} then, undoing the exponential tilt, one can directly check that the function
\begin{equation}
  \label{26may510}
  w(t, y) \coloneqq \abs{t}^{\frac{d-1}{2}} \e^{-y} v(t, y)
\end{equation}
satisfies the one-dimensional problem 
\begin{equation}
  \label{26apr2902}
  \begin{cases}
    \partial_t w = \partial_y^2 w + \dot{b}_{\abs{t}} \partial_y w + w, & y > 0,\\
    w(t,0) = \abs{t}^{\frac{d-1}{2}}.
  \end{cases}
\end{equation}
As a consequence, a direct computation shows that the superposition
\begin{equation}
  \label{eq:superposition-necessary}
  W(t, x) \coloneqq \int_{S^{d-1}} w(t, b_{\abs{t}} - \theta \cdot x) \d \theta, \quad t\leq -1, x\in\R^d, |x|\le b_{|t|}
\end{equation}
satisfies the linear equation 
\begin{equation}
  \label{26may506}
  \partial_t W = \Delta W + W \quad \text{where } t < -1,~ x\in\R^{d},~\abs{x} < b_{\abs{t}}.
\end{equation}
Now fix $x$ with $\abs{x} = b_{\abs{t}}$ and let $\Gamma_x$ denote the ball in $S^{d-1}$ 
of radius $\abs{t}^{-1/2}$ about the point $\h x=x/|x|.$
For $\theta \in \Gamma$, we have
\begin{equation*}
  0 \leq b_{\abs{t}} - \theta \cdot x \lesssim 1.
\end{equation*}
By a boundary Harnack estimate for \eqref{26apr2902}, we deduce that  
\begin{equation}
  \label{26apr2904}
  w(t, b_{\abs{t}} - \theta \cdot x) \gtrsim w(t, 0) = \abs{t}^{\frac{d-1}{2}} \For \theta \in \Gamma_x.
\end{equation}
Since 
\begin{equation*}
  \op{vol}_{S^{d-1}}(\Gamma_x) \asymp \abs{t}^{-\frac{d-1}{2}}
\end{equation*}
and $w > 0$ (by the maximum principle), we obtain from \eqref{eq:superposition-necessary} and \eqref{26apr2904} that 
\begin{equation}
  \label{26may508}
  W(t, x) \gtrsim 1 \quad \text{for any } x\in\R^d \text{ with } |x|=b_{|t|}.
\end{equation}
It follows from \eqref{26may506} and \eqref{26may508}
that  $CW$ is a super-solution of \eqref{26may423} for $C \gg 1$.
By the comparison principle, we deduce that 
\begin{equation}
  \label{26may512}
  h(t, x) \lesssim W(t, x) \quad \text{for } \abs{x} \leq b_{\abs{t}}.
\end{equation}
By \eqref{26apr2804}, \eqref{26may510} and \eqref{26may512}, 
for $t \leq \tmax$ we have
\begin{equation*}
  h(t, 0) \lesssim \abs{t}^{\frac{d-1}{2}} (b_{\abs{t}} + 1) 
  \e^{-b_{\abs{t}}} = \abs{t}^{\m{O}(1)} \e^{-2\abs{t}}\to 0~\hbox{ as $t\to-\infty$}.
\end{equation*}
This proves \eqref{26may416} and finishes the proof of Proposition \ref{prop:necessary}.
We can now prove the negative part of Theorem~\ref{thm:LIL}.
\begin{proof}[Proof of second part of Theorem~\ref{thm:LIL}]
  Consider a BBM started from the origin at time $0$.
  Let $A$ denote the event that there is a sequence of times $(t_n)_{n \in \N}$ tending to infinity at which a particle has radius $b_{t_n}$.
  Then
  \begin{equation*}
    A = \bigcap_{T \geq 1} E_T(S^{d-1}).
  \end{equation*}
  By Proposition~\ref{prop:necessary}, $\P(A) = 0$.
  If $F$ is the event described in the second part of Theorem~\ref{thm:LIL}, then we can write $F = A^\cc \setminus B$, where $B$ is the event that there exists a particle whose descendants remain forever outside the radius $b$.
  However, Mallein~\cite{Mallein15} showed that the most distant particle typically has radius $2t + \frac{d-4}{2} \log t + \m{O}(1)$, which is much less than $b_t$.
  Hence $\P(B) = 0$ and $\P(F) = \P(A^\cc) = 1$.
\end{proof}

\section{Total mass}
\label{sec:total}

The remaining sections have a somewhat different flavor, as they are concerned with consequences of Theorem~\ref{thm:LIL} rather than its proof.
Throughout, we will assume that $\beta$ satisfies the hypotheses of Theorem~\ref{thm:LIL} and \eqref{eq:integral-condition}.
Given an open sector $U \subset S^{d-1}$ and $T \geq 2$, Theorem~\ref{thm:LIL} provides a first time $\h t \geq T$ and particle $\h p \in \m{P}_{\h t}$ that crosses the radial threshold $r = b_{\h t}$ through the sector $U$ after time $T$.
We can then define the modified reference frame
\begin{equation}
  \label{26may520}
  \h{b}_t \coloneqq 2t + \frac{d-1}{2} \log t - \frac{1}{2} \log \h t + \beta_{\h t}, \quad t\ge \h t
\end{equation}
as in \eqref{eq:b-hat}, and the restricted approximate derivative martingale measure
\begin{equation}
  \label{26may518}
  \h{Y}_t(\dn\theta) \coloneqq \pi^{\frac{d-1}{4}} \sum_{p \in \h{\m{P}}_t} ([\h{b}_t - R_t(p)]_+ + 1) \e^{-[\h{b}_t - R_t(p)]} \delta_{\Theta_t(p)}(\dn\theta), \quad \text{for } t \geq \h t.
\end{equation}
We recall that $\h{\m{P}}_t$ denotes the descendants of $\h p$ at time $t$.
Also, $\m{G}_{U,T}$ denotes the $\sigma$-algebra containing all information about particles until they first cross the radial threshold $b$ through the sector $U$ after time $T$.

Applying McKean's formula~\cite{McKean75} to \eqref{26may518}, we can express the Laplace transform 
of the total mass $\h{Y}_t(S^{d-1})$ conditional on $\m{G}_{U,T}$ in terms of the Fisher--KPP equation.
For each $\lambda > 0$ and $t \geq \h t$, let $w(s,r;t,\lambda)$, be the  radial solution to the backward Fisher--KPP equation
\begin{equation}
  \label{eq:radial-KPP}
  \begin{cases}
    -\partial_s w = \partial_r^2 w + \frac{d-1}{r} \partial_r w + w(1 - w), \quad s<t,\\[4pt]
    w(s = t, r; t, \lambda) = 1 - \exp\bigl(-\pi^{\frac{d-1}{4}}\lambda [(\h b_t - r)_+ + 1] \e^{-[\h b_t - r]}\bigr).
  \end{cases}
\end{equation}
Then, the Laplace transform of $\h {Y}_t(S^{d-1})$ has the form 
\begin{equation}
  \label{eq:McKean-total}
  \E \exp[-\lambda \h{Y}_t(S^{d-1}) \mid \m{G}_{U,T}] = 1-w(\h t, b_{\h t}\,; t, \lambda).
\end{equation}
We condition on $\m{G}_{U,T}$ so we can take $t \geq \h t$ outside the expectation.

We bound $w$ from below using a one-dimensional traveling wave.
Let $\phi$ be the 1D Fisher--KPP traveling wave of speed $2$ connecting $1/2$ to $0$ for the nonlinearity $f(u) = u(1 - 2u)$, with $\phi(0) = 1/4$:
\begin{equation}
  \label{26may602}
  -2\phi'=\phi''+\phi(1-2\phi), \quad \phi(-\infty)=\frac{1}{2},\enspace \phi(+\infty)=0, \enspace\phi(0)=\frac{1}{4}.
\end{equation}
The following lemma gives a subsolution for \eqref{eq:radial-KPP}.
\begin{lemma}
  \label{lem:wave-lower}
  There exists a constant $\ubar{C}(d) > 0$ and a uniformly bounded $\m{C}^1$ function $\xi \colon [2, t] \to [0, \infty)$ such that for all $s \in [\ubar{C}, t]$ and $\lambda \geq 1$,
  \begin{equation}
    \label{eq:sub}
    \ubar{w}(s, r; t, \lambda) \coloneqq \phi\bigl(\h b_s + \xi_s - \frac{1}{2} \log \lambda - r\bigr)
  \end{equation}
  satisfies 
  \begin{equation}
    \label{26may625}
    -\partial_s \ubar w \leq \partial_r^2 \ubar w + \frac{d-1}{r} \partial_r \ubar w + \ubar w (1 - \ubar w).
  \end{equation}
\end{lemma}
\begin{proof}
  Differentiating \eqref{26may520} we see that
  \begin{equation}
    \label{26may604}
    \dot{\h{b}}_s = 2 + \frac{d-1}{2s}.
  \end{equation}
  We compute, using \eqref{26may602} and \eqref{26may604}, 
  \begin{align}
    \m{N}[\ubar{w}] &\coloneqq \partial_s \ubar w + \partial_r^2 \ubar w + \frac{d-1}{r} \partial_r \ubar w + \ubar w (1 - \ubar w)\nonumber\\
                    &=\Big(2 + \frac{d-1}{2s}+\dot\xi_s	-\frac{d-1}{r}-2\Big)\phi'-\phi(1-2\phi)+\phi(1-\phi)\nonumber\\
                    &= \Big[\dot \xi_s + (d-1)\Big(\frac{1}{2 s} - \frac{1}{r}\Big)\Big] \phi' + \phi^2.\label{eq:nonlin}
  \end{align}
  Let us take  
  \begin{equation*}
    \xi_s=\int_s^t \frac{d(d-1)\log s_1}{8s_1^2}ds_1.
  \end{equation*}
  As we are only interested in $s\ge \h t\ge 2$, we know that $\xi_s$ is uniformly bounded:
  \begin{equation}
    \label{26may614}
    \sup_\tau\xi_\tau<+\infty.
  \end{equation}
  Because $\phi' < 0,$ we have
  \begin{equation}
    \label{26may618}
    \m{N}[\ubar{w}] \geq 0,\
  \end{equation}
  wherever
  \begin{equation}
    \label{26may610}
    \frac{d(d-1)\log s}{8s^2} - (d-1)\left(\frac{1}{2 s} - \frac{1}{r}\right) \geq 0.
  \end{equation}
  One can check directly that \eqref{26may610} holds if  
  \begin{equation*}
    r \leq 2 s + \frac{d}{2} \log s - 1,
  \end{equation*}
  provided $s \gg 1$.

  On the other hand, when 
  \begin{equation}
    \label{26may620}
    r > 2 s + \frac{d}{2} \log s - 1,
  \end{equation}
  we see from \eqref{26may520} that the argument of $\phi$ in \eqref{eq:sub} satisfies
  \begin{align*}
    \h b_s + \xi_s - \frac{1}{2} \log \lambda - r &\leq -\frac{1}{2} \log s - \frac{1}{2} \log \h t + \beta_{\h t} + \xi_s - \frac{1}{2} \log \lambda + 1\\
                                                  &\leq -\frac{1}{2} \log s + \sup_{\tau \geq 2} \left(\beta_\tau - \frac{1}{2} \log \tau\right) + \sup \xi + 1 \to -\infty
  \end{align*}
  as $s \to \infty$.
  Here we used $\lambda \geq 1$, as well as \eqref{eq:beta-bounds} and \eqref{26may614} to ensure the suprema are finite.
  Because $\phi(-\infty) = 1/2 > 0$ in \eqref{26may602}, the second term in \eqref{eq:nonlin} is uniformly positive in this limit.
  In contrast, the first term in \eqref{eq:nonlin} tends to zero in this regime, because $\phi'$ is uniformly bounded.
  It follows that for sufficiently large $s$, \eqref{26may618} holds also in the region \eqref{26may620}.
  That is, there exists $\ubar{C}(d) \gg 1$ such that $\m{N}[w] \geq 0$ for all $r \in \R$ and $s \geq \ubar{C}$.
\end{proof}
\begin{proof}[Proof of Lemma~\ref{lem:total}]
  Recall that if a solution of a forward-in-time Fisher--KPP is near $1$ on a sufficiently large ball
  initially, it will remain near $1$ at the center for all times.
  To see this, one can construct a suitable compactly supported sub-solution.
  The terminal data 
  \begin{equation}
    \label{26may621}
    w(s = t, r; t, \lambda) = 1 - \exp\bigl(-\pi^{\frac{d-1}{4}}
    \lambda [(\h b_t - r)_+ + 1] \e^{-[\h b_t - r]}\bigr)
  \end{equation}
  in the backward-in-time Fisher--KPP equation \eqref{eq:radial-KPP} tends to $1$ locally uniformly in $\h b_t - r$ as $\lambda \to \infty$.
  The aforementioned property implies that
  there exists a constant $C(d) > 0$ such that for all $\lambda \geq C$ and $s \leq t$, we have
  \begin{equation}
    \label{26may624}
    w(s, \h b_t) \geq \frac{1}{2} = \sup \ubar{w}.
  \end{equation}
  In addition, the speed-2 Fisher--KPP traveling wave $\phi$ defined in \eqref{26may602} satisfies
  \begin{equation}
    \label{26may623}
    \phi(a) \lesssim a \e^{-a} \As a \to \infty.
  \end{equation}
  On the other hand, the terminal condition \eqref{26may621} satisfies
  \begin{equation}
    \label{26may622}
    w(s = t, r; t, \lambda) \sim  \pi^{\frac{d-1}{4}}
    \lambda (\h b_t - r) \e^{-(\h b_t - r)} \For r\ll \h b_t.
  \end{equation}
  Comparing \eqref{26may623} and~\eqref{26may622},
  we see that we can choose $C(d) \gg 1$ such that for all $\lambda \geq C(d)$,
  \begin{equation}
    \label{eq:sub-initial}
    w|_{s=t} \geq \ubar{w}|_{s=t} \ForAll r \leq \h b_t.
  \end{equation}
  With ordered boundary and terminal data in \eqref{26may624} and \eqref{eq:sub-initial}, the comparison principle together with \eqref{eq:radial-KPP} and \eqref{26may625}
  implies that 
  \begin{equation}
    \label{26may626}
    w \geq \ubar{w} \ForAll s \in [\ubar{C}, t],r \leq \h b_t.
  \end{equation}

  We first suppose 
  \begin{equation}
    \label{26may630}
    t \geq \h t + \frac{1}{8}\log \lambda + 1
  \end{equation}
  and take
  \begin{equation*}
    s= \h t + \frac{1}{8}\log \lambda.
  \end{equation*}
  We assume $\lambda$ is chosen large enough that $s \geq \ubar{C}$ and $b_{\h{t}} \leq b_s$.
  Then, we can apply \eqref{26may626} at $r=b_{\h t}$ to obtain
  \begin{align}
    w\big(s, b_{\h t}) &\geq \ubar{w}(s, b_{\h t}) = \phi\bigl(\h b_s + \xi_s - \frac{1}{2} \log \lambda - b_{\h t}\bigr)\nonumber\\
                       &= \phi\Big(2s + \frac{d-1}{2} \log s - \frac{1}{2} \log \h t + \beta_{\h t}- 2\h t - \frac{d - 2}{2} \log \h t - \beta_{\h t}+ \xi_s - \frac{1}{2} \log \lambda\Big)\nonumber\\
                       &=\phi\Big(\!-\frac{1}{4}\log\lambda + \frac{d-1}{2} \big[\log(\h t+  \frac{1}{8}\log \lambda) -\log \h t\big]+\xi_s\Big)\label{26may628}
  \end{align}
  Now $\h t \geq 2$ and \eqref{26may614} yield
  \begin{align*}
    -\frac{1}{4}\log\lambda + \frac{d-1}{2} \big[\log(\h t+  &\frac{1}{8}\log \lambda) -\log \h t\big]+\xi_s\\
                                                             &\leq -\frac{1}{4}\log\lambda + \frac{d-1}{2} \log\left(2^{-4} \log \lambda + 1\right) + \sup \xi.
  \end{align*}
  This tends to $-\infty$ as $\lambda \to \infty$, so we can choose $\lambda \geq C$ sufficiently large (and independent of $\h t$) so that \eqref{26may628} yields
  \begin{equation}
    \label{eq:trigger}
    w\big(s, b_{\h t}) \geq \phi(0) = \frac{1}{4}.
  \end{equation}
  Parabolic regularity and \eqref{eq:trigger} imply that there is a constant $c(d) > 0$ 
  such that 
  \begin{equation*}
    w\big(\h t + \tfrac{1}{8}\log \lambda,r) \geq c \quad \text{where } \abss{r - b_{\h t}} \leq 1.
  \end{equation*}
  By the hair trigger effect, the solution of the forward Fisher--KPP equation
  with such data converges to $1$ locally 
  uniformly as time tends to infinity.
  Evolving backward from $s=\h t + ({1}/{8}) \log \lambda$ to $\h t$ accomplishes the same, so 
  \begin{equation}
    \label{eq:unif}
    w(\h t, b_{\h t}) \to 1 \As \lambda \to \infty \text{ uniformly in }t, \h t.
  \end{equation}

  Next, consider the opposite situation to \eqref{26may630}:
  \begin{equation}
    \label{26may632}
    t \leq \h t + \frac{1}{8}\log \lambda + 1.
  \end{equation}
  Then, as in the computation in \eqref{26may628}, we have 
  \begin{equation*}
    \h b_t - b_{\h t} \leq \frac{1}{4} \log \lambda + \m{O}(\log \log \lambda) \leq \frac{1}{2} \log \lambda
  \end{equation*}
  provided $\lambda \geq C \gg 1$.
  Thus, the terminal condition \eqref{26may621} gives
  \begin{align*}
    w(t,b_{\h t} + r; t, \lambda)&= 1 - \exp\bigl(-\pi^{\frac{d-1}{4}}
                                   \lambda [(\h b_t -b_{\h t}- r)_+ + 1] \e^{-[\h b_t -b_{\h t}- r]}\bigr)\\
                                 &\geq 1 - \exp\left(-\pi^{\frac{d-1}{4}} \sqrt{\lambda} \e^r\right).
  \end{align*}
  In particular, $w|_{s = t}$ is close to $1$ on a large ball centered at $b_{\h t}$ if $\lambda$ is sufficiently large.
  As indicated earlier, this implies that $w$ will remain close to $1$ at all earlier times, including at $\h t$.
  Here ``close'' and ``large'' are determined by $\lambda$ alone.
  Thus \eqref{eq:unif} also holds in the regime \eqref{26may632}, and is therefore true unconditionally.

  By \eqref{eq:McKean-total}, we deduce that
  \begin{equation*}
    \E \exp[-\lambda \h Y_t(S^{d-1}) \mid \m{G}_{U,T}] \to 0 \quad \text{as } \lambda \to \infty
  \end{equation*}
  uniformly in $t$.
  The conclusion of Lemma~\ref{lem:total} follows.
\end{proof}

\section{Localization: the proof of Proposition~\ref{prop:local}}
\label{sec:local}

In this section, we prove Proposition~\ref{prop:local}, which states that most of the mass of $\h Y$ remains confined to the vicinity of $\h \theta$, the angle of the particle $\h p$ at the time $\h t \geq T$ of its excursion.
Precisely, let $\Gamma$ be the ball in $S^{d-1}$ centered at $\h \theta$ of radius
\begin{equation}
  \label{eq:Gamma-rad}
  \rho_\Gamma \coloneqq A \h t^{-1/2} \log^{1/2} \h t,
\end{equation}
for a constant $A$ to be determined.
We must show that for all $K > 0$, there exist $A(K),C(K) > 0$ such that for all $T \geq C$ and~$t \geq \h t$, we have
\begin{equation}
  \label{26may702}
  \P\big[\h Y_t(\Gamma^\cc) \geq \h t^{-K} \mid \m{G}_{U,T}\big] \leq T^{-K}.
\end{equation}
We know from Lemma~\ref{lem:total} that the total mass of $\h Y_t$ is uniformly positive with high probability.
Here, we wish to show that $\h Y_t$ places most of its mass in $\Gamma$ as $t \to \infty$.
As in the previous section, we condition on $\m{G}_{U,T}$ so we can take $t \geq \h t$.

We once again use McKean's formula to express the problem in terms of the Fisher--KPP equation.
This is essentially identical to \eqref{eq:radial-KPP} and \eqref{eq:McKean-total}, except now we only look at the mass in $\Gamma^\cc$ rather than the whole sphere.
Given $t > \h t$ and $x = (r, \theta) \in \R^d$, define the profile
\begin{equation}
  \label{eq:lacuna-data}
  \phi(x; t) = \phi(r, \theta; t) \coloneqq \pi^{\frac{d-1}{4}} [(\h b_t - r)_+ + 1] \e^{-(\h b_t - r)} \tbf{1}_{\Gamma^\cc}(\theta).
\end{equation}
Then for all $\lambda > 0$, we have
\begin{equation*}
  \E \exp[-\lambda \h Y_t(\Gamma^\cc) \mid \m{G}_{U,T}] = \E\Big[\prod_{p \in \h{\m{P}}_t} \exp\big[\!-\lambda \phi\big(X_t(p); t\big)\big] \;\big|\; \m{G}_{U,T}\Big].
\end{equation*}
As in \eqref{eq:McKean-total}, McKean's formula implies that
\begin{equation}
  \label{eq:McKean-again}
  1 - \E \exp[-\lambda \h Y_t(\Gamma^\cc) \mid \m{G}_{U,T}] = u(\h t, b_{\h t}, \h\theta; t, \lambda)
\end{equation}
for $u(s,r,\theta; t, \lambda)$ solving a backward Fisher--KPP problem in $(-\infty, t] \times \R^d$:
\begin{equation}
  \label{eq:backward-multi}
  \begin{cases}
    \partial_s u + \Delta u + u (1 - u) = 0,\\
    u(s = t, x; t, \lambda) = 1 - \exp[-\lambda \phi(x; t)]
  \end{cases}
\end{equation}

To prove \eqref{26may702}, we must show that $u$ is small at the center $\h \theta$ of $\Gamma$ when $\lambda = \h t^K$.
To control $u$ from above, we take superpositions of 1D linearized solutions with Dirichlet boundary data.
We work in a frame moving at speed $2$ with logarithmic correction of the form
\begin{equation}
  \label{eq:drift}
  g(s) \coloneqq \frac{d-1}{2} \log s.
\end{equation}
In the absence of $g$, one can write down simple solutions of the linear problem, which is the 1D heat equation modulo an exponential tilt.
We first show that the modest drift $g$ does not spoil things.
\begin{lemma}
  \label{lem:drift-heat}
  Let $v(s, y; t)$ solve 
  \begin{equation*}
    \begin{cases}
      \partial_s v + \partial_y^2  v + \dot{g} \partial_y  v = 0, & s \in (1,t), y > 0,\\
      v(s = t, y) = y,\\
      v(s,0) = 0.
    \end{cases}
  \end{equation*}
  Then there exists $\smin(d) > 0$ such that 
  \begin{equation}
    \label{26may710}
    y \leq v(s, y; t) \leq 2y \ForAll s \in [\smin, t], y \geq 0.
  \end{equation}
\end{lemma}
\begin{proof}
  To frame this as a forward problem, we set 
  $\ti v(s,y)=v(t-s,y)$.
  Then
  \begin{equation*}
    \begin{cases}
      \partial_s \ti v = \partial_y^2 \ti v + \kappa(s) \partial_y  \ti v, & s \in (0, t - \smin), y > 0,\\
      \ti v(0, y) = y,\\
      \ti v(s,0) = 0.
    \end{cases}
  \end{equation*}
  with 
  \begin{equation}
    \label{26may714}
    \kappa(s) \coloneqq \frac{d-1}{2(t-s)} > 0.
  \end{equation}
  Then $w \coloneqq  \ti v - y$ satisfies
  \begin{equation}
    \label{eq:1D-diff}
    \partial_s w = \partial_y^2 w + \kappa(\partial_y w + 1), \quad w(0, y) = 0, \enspace w(s, 0) = 0.
  \end{equation}
  Because $\kappa > 0$, the comparison principle implies that $w > 0$.
  This gives the lower bound in \eqref{26may710}, and it thus suffices to control $w$ from above.

  Let $z$ solve the simpler problem 
  \begin{equation}
    \label{eq:simpler}
    \partial_s z = \partial_y^2 z + \kappa, \quad z(0,y) = 0, \enspace z(s,0) = 0.
  \end{equation}
  Using the Duhamel formula, we can express $z$ in terms of the error function $\Phi$ from \eqref{eq:erf}:
  \begin{equation*}
    z(s, y) = \int_0^s \kappa(s') \Phi\left(\frac{y}{\sqrt{s-s'}}\right) \d s'.
  \end{equation*}
  Differentiating in $y$, and using \eqref{26may714}, we find
  \begin{equation}
    \label{26may716}
    \partial_y z \lesssim \int_0^s (t - s')^{-1} (s - s')^{-1/2} \d s' = \int_0^{s} (t - s + r)^{-1} \frac{\dn r}{\sqrt{r}}.
  \end{equation}
  If $t-s>s$, then the integral in the right side above can be estimated as
  \begin{equation}
    \label{eq:small-s}
    \partial_y z \lesssim (t-s)^{-1}\sqrt{s} \lesssim (t - s)^{-1/2}.
  \end{equation}
  On the other hand, if $t-s<s$, then the right side of \eqref{26may716} is bounded similarly by 
  \begin{equation}
    \label{26may718}
    \begin{aligned}
      \partial_y z & \lesssim\int_0^{t-s} (t - s + r)^{-1} \frac{\dn r}{\sqrt{r}} +\int_{t-s}^s (t - s + r)^{-1} \frac{\dn r}{\sqrt{r}}\\
                   &\lesssim (t-s)^{-1}\sqrt{t-s}+ \int_{t-s}^{\infty} r^{-3/2} {\dn r} 
                     \lesssim (t - s)^{-1/2}.
    \end{aligned}
  \end{equation}
  Recalling that $s<t-\smin$, it follows that $\partial_y z \leq 1$ provided $\smin \gg 1$.
  Comparing \eqref{eq:1D-diff} and \eqref{eq:simpler}, we see that $2z$ is then a supersolution for \eqref{eq:1D-diff}.
  Hence \eqref{eq:small-s} and \eqref{26may718} yield
  \begin{equation*}
    v(s,y)=y+w(t-s,y) \leq y+2z(t-s,y) \leq [1 + \m{O}(\smin^{-1/2})] y.
    \qedhere
  \end{equation*}
\end{proof}
We now construct a putative super-solution for \eqref{eq:backward-multi} from an angular superposition of $v$ with exponential tilt.
Recalling the radius $\rho_\Gamma$ of $\Gamma$ in $S^{d-1}$ from \eqref{eq:Gamma-rad}, we define
\begin{equation}
  \label{26may722}
  m(\theta) \coloneqq \max\big\{\big[\dist_{S^{d-1}}(\theta, \h\theta) - \rho_\Gamma\big]_+^{-1/2},\, 1\big\}.
\end{equation}
We use this as a density on $\Gamma^\cc$ that places extra weight near $\partial \Gamma$.
Define
\begin{equation}
  \label{26may806}
  \h \gamma \coloneqq \frac{1}{2} \log \h t - \beta_{\h t}
\end{equation}
and
\begin{equation}
  \label{eq:superposition}
  V(s, x; t) \coloneqq \int_{\Gamma^\cc} v(s, 2s + g(s) - \h \gamma - \theta \cdot x + 1; t) \e^{-(2 s - \h \gamma - \theta \cdot x)} m(\theta) \d \theta.
\end{equation}
We can easily check that
\begin{equation}
  \label{26may719}
  \ti u(s, x) \coloneqq v(s, 2 s + g(s) - \h\gamma - \theta \cdot x) \e^{-(2 s - \h\gamma - \theta \cdot x)}
\end{equation}
solves the linearized Fisher--KPP equation
\begin{equation*}
  \partial_s \ti u + \Delta \ti u + \ti u = 0
\end{equation*}
where the argument of $v$ in \eqref{26may719} is positive.
Integrating over $\theta \in \Gamma^\cc$ with weight $m$, the same is true of $V$:
\begin{equation}
  \label{26may802}
  \partial_s V + \Delta V + V = 0.
\end{equation}
Thus, $V$ is potentially a super-solution for $u$ that solves \eqref{eq:backward-multi}.
However, we must be careful with the domain of definition of $V$, as $v$ is only defined on the half-line.
We will not use the values of $v$ very close to the origin, so we focus on its restriction to $[1, \infty)$.
The superposition $V$ is well-defined on
\begin{equation}
  \label{26may723}
  D \coloneqq \{(s, x) \in [1, t] \times \R^d \colon \theta \cdot x < 2 s + g(s) - \h\gamma ~\text{ for all }\theta \in \Gamma^\cc\}.
\end{equation}
Given $s \in [1, t]$, we let $D(s)$ denote the $s$ time-slice of $D$, which is the convex subset of $\R^d$ formed from the intersection of the half-spaces $\{\theta \cdot x < 2 s + g(s) - \h\gamma\}$ over all $\theta \in \Gamma^\cc$.
Let $\spartial D$ denote the \emph{spatial} boundary of $D$, namely the union of $\{s\} \times \partial D(s)$ over $s \in [1, t]$.
\begin{lemma}
  \label{lem:large-enough}
  There exists a constant $c(d) > 0$ such that 
  \begin{equation*}
    \inf_{(s,x)\in\spartial D} V(s,x;t) \geq c.
  \end{equation*}
\end{lemma}
\begin{proof}
  Fix $s \in [1, t]$ and $x_0 = (r_0, \theta_0) \in \partial D(s)$, so $r_0 \asymp s$.
  We apply Laplace's method to \eqref{eq:superposition}, so the estimate hinges on the maximizer of the function $\eta(\theta) \coloneqq \theta \cdot \theta_0$ for $\theta\in\Gamma^{\cc}$, which implicitly appears in the exponent in \eqref{eq:superposition} for $V(s,x_0;t)$.

  We first consider $\theta_0 \in \Gamma$.
  Then there exists 
  \begin{equation}
    \label{26may724}
    \theta_1 \in \argmax_{\theta\in\Gamma^{\cc}} \eta(\theta) \subset \partial \Gamma,
  \end{equation}
  and the maximizer is unique unless $\theta_0 = \h\theta$.
  Near $\theta_1$, the function $\eta$ falls linearly in the direction $e_\parallel \in T_{\theta_1}S^{d-1}$ connecting $\theta_0$ to $\theta_1$, 
  and quadratically in the $d-2$ orthogonal directions $E_\perp \subset T_{\theta_1}S^{d-1}$.
  Observe that $r_0 \asymp s$ and 
  \begin{equation*}
    \e^{\theta \cdot x} = \e^{r_0 \eta(\theta)}.
  \end{equation*}
  It follows that the function $\theta \mapsto \e^{\theta \cdot x}$ falls by no more than a constant factor when we move distance $\asymp s^{-1}$ along $e_\parallel$ from $\theta_1$, and $\asymp s^{-1/2}$ in $E_\perp$.
  Due to the square-root weight in the definition \eqref{26may722} of $m$, 
  this region in $\Gamma^\cc$ has $m$-measure of order 
  \begin{equation}
    \label{26may1104}
    s^{1/2}s^{-1}s^{-(d-2)/2} = s^{-(d-1)/2}.
  \end{equation}
  It follows that
  \begin{equation}
    \label{eq:Laplace}
    V(s, x_0; t) \gtrsim s^{-(d-1)/2} v\big(s, 2 s + g(s) - \h\gamma - \theta_1 \cdot x_0 + 1; t\big) \e^{-(2s - \h\gamma - \theta_1 \cdot x_0)}.
  \end{equation}
  Since $x_0\in\partial D$, we see from the definition \eqref{26may723} of the domain $D$ and \eqref{26may724} that $x_0$ lies in the plane
  \begin{equation}
    \label{26may725}
    \{\theta_1 \cdot x = 2 s + g(s) - \h\gamma\}.
  \end{equation}
  Thus Lemma~\ref{lem:drift-heat} yields
  \begin{equation*}
    v(s, 2 s + g(s) - \h\gamma - \theta_1 \cdot x_0 + 1; t) = v(s, 1; t) \asymp 1,
  \end{equation*}
  while \eqref{eq:drift} and \eqref{26may725} imply that 
  \begin{equation*}
    \e^{-(2s - \h\gamma - \theta_1 \cdot x_0)} = \e^{g(s)} = s^{(d-1)/2}.
  \end{equation*}
  Combining these in \eqref{eq:Laplace}, we obtain 
  \begin{equation*}
    V(s, x_0; t) \gtrsim 1.
  \end{equation*}
  Conversely, if $\theta_0 \in \Gamma^\cc$, then $\eta$ is maximized at $\theta_1 = \theta_0$ 
  and its values fall quadratically in all directions.
  So $\theta \mapsto \e^{\theta \cdot x}$ falls by no more than a constant factor in a region of radius $s^{-1/2}$ around $\theta_0$ in $\Gamma^\cc$.
  Because $m \geq 1$, this region has $m$-measure $\gtrsim s^{-(d-1)/2}$ and
  \begin{equation}
    \label{eq:standard-Laplace}
    V(s, x_0; t) \gtrsim s^{-(d-1)/2} v\big(s, 2 s + g(s) - \h\gamma - \theta_1 \cdot x_0 + 1\big) \e^{-(2s - \h\gamma - \theta_1 \cdot x_0)} \asymp 1.
  \end{equation}
  Here we use the fact that $x_0$ still satisfies \eqref{26may725} because it lies on $\partial D$.
  The implied constants depend only on the dimension $d$.
\end{proof}
As a consequence of Lemma~\ref{lem:large-enough}, we know that a sufficiently large multiple of $\lambda V$ exceeds the boundary data of $u$ on $\spartial D$.
Next, we show that such a multiple also exceeds the terminal data of $u$ in $D(t)$.
\begin{lemma}
  \label{lem:super-initial}
  There exists $C(d) > 0$ such that 
  \begin{equation}
    \label{26may814}
    C \lambda V(t,x;t) \geq u(t, x;t) \ForAll x\in D(t).
  \end{equation}
\end{lemma}
\begin{proof}
  We first observe that 
  \begin{equation*}
    u(s=t,x;t,\lambda)=1-e^{-\lambda\phi(x;t)} \leq \lambda \phi(x;t).
  \end{equation*}
  Thus it suffices to show that 
  \begin{equation}
    \label{26may804}
    V(t, x) \gtrsim \phi(x) \For x\in D(t).
  \end{equation}
  Let $x_0=(r_0,\theta_0) \in D(t)$.
  The definition \eqref{eq:lacuna-data} of $\phi$ implies that $\phi = 0$ for $\theta_0 \in \Gamma$.
  Thus, we only need to check \eqref{26may804} for $\theta_0 \in \Gamma^\cc$.
  In that case, we may take $\theta=\theta_0$ in the definition \eqref{26may723} of $D$.
  Recalling \eqref{eq:b-hat}, \eqref{eq:drift}, and \eqref{26may806}, this gives
  \begin{equation}
    \label{26may812}
    r_0\leq 2t+g(t)- \h\gamma =2t+\frac{d-1}{2}\log t-\frac{1}{2}\log\h t+\beta_{\h t}= \h b_t.
  \end{equation}
  We now write \eqref{eq:superposition} as 
  \begin{equation}
    \label{26may808}
    \begin{aligned}
      V(t,x_0; t) &
                    =\e^{-[\hat b_t - g(t)]}\int_{\Gamma^\cc} v(t, \h b_t - r_0\theta_0 \cdot\theta + 1; t) 
                    \e^{r_0\theta_0\cdot \theta} m(\theta) \d \theta.
    \end{aligned}
  \end{equation}
  As in the second part of the proof of Lemma \ref{lem:large-enough}, since $\theta_0\in\Gamma^\cc$, the maximizer of $\theta_0\cdot\theta$ over $\theta \in \Gamma^\cc$ is $\theta_0$ itself.
  Then, as in \eqref{eq:standard-Laplace} but with $s=t$, \eqref{26may812} and \eqref{26may808} yield
  \begin{equation*}
    V(t,x_0; t) \gtrsim  t^{-(d-1)/2} \e^{-[\hat b_t - g(t)]} v\big(t, \h b_t - r_0 + 1; t) \e^{r_0}.
  \end{equation*}
  Using Lemma~\ref{lem:drift-heat} and \eqref{eq:drift}, we get
  \begin{equation*}
    V(t, r_0, \theta_0; t) \gtrsim \pi^{\frac{d-1}{4}} (\h b_t - r_0 + 1) \e^{-(\h b_t - r_0)} = \phi(r_0, \theta_0; t),
  \end{equation*}
  and \eqref{26may814} follows.
\end{proof}
We can now control $u$ in the center of the lacuna $\Gamma$ in its terminal data $\phi$.
Recall that $\Gamma$ is centered at the angle $\h\theta$ of the excursion, and its radius in \eqref{eq:Gamma-rad} is
\begin{equation*}
  \rho_\Gamma=A \h t^{-1/2} \log^{1/2} \h t \gg \hat t^{-1/2}.
\end{equation*}
\begin{lemma}
  \label{lem:lacuna-small}
  There exists a constant $C(d) > 0$ such that for all $t \geq \h t$ and $\lambda > 0$,
  \begin{equation}
    \label{26may1116}
    u(\h t, b_{\h t}, \h\theta; t, \lambda) \leq C(d) \lambda A^2 (\log \h t)^{d/2} \, \h t^{-A^2}.
  \end{equation}
\end{lemma}
\begin{proof}
  First, Lemmas~\ref{lem:large-enough} and \ref{lem:super-initial}, together with \eqref{26may802}, imply that the function
  \begin{equation*}
    \bar{V} \coloneqq
    \begin{cases}
      \min\{C \lambda V, 1\} & \text{in } D,\\
      1 & \text{in } D^\cc
    \end{cases}
  \end{equation*}
  is a supersolution of \eqref{eq:backward-multi} for $C(d) \gg 1$ that lies above $u$ at the boundary, whence
  \begin{equation*}
    u \leq \bar{V}.
  \end{equation*}
  The point $(\h t,\h x)$ with  $\h x \coloneqq (b_{\h t}, \h \theta)$ lies in $D$ because for any $\theta\in\Gamma^\cc$,
  \begin{equation}
    \label{26may1106}
    \theta\cdot \h x<\h \theta \cdot \h x =b_{\h t}=2\h t +\frac{d-2}{2}\log\h t+\beta_{\h t}=2 \h t + g(\h t) - \h\gamma.
  \end{equation}
  Hence, 
  \begin{equation}
    \label{26may1114}
    u(\h t, \h x;t) \le C \lambda V(\h t,\h x;t),
  \end{equation}
  and it suffices to control $V$ from above.
  We will once again use the Laplace method on the expression
  \begin{equation}
    \label{26may1102}
    V(\h t, \h x; t) \coloneqq \int_{\Gamma^\cc} v(\h t, 2\h t + g(\h t) - \h \gamma - \theta \cdot \h x + 1; t) \e^{-(2 \h t - \h \gamma - \theta \cdot \h x)} m(\theta) \d \theta.
  \end{equation}
  The function $\theta \mapsto \theta \cdot \h \theta$ is maximized over $\Gamma^\cc$ along the entire boundary $\partial \Gamma$.
  Because $b_{\h t} \asymp \h t$, the function $\e^{\theta \cdot \h x}$ falls rapidly in $\Gamma^\cc$ outside the collar of width $\h t^{-1}$ around $\partial \Gamma$.
  Due to the square-root weight in the definition \eqref{26may722} of $m$, and as in the computation leading to \eqref{26may1104}, this collar has $m$-measure of order
  \begin{equation}
    \label{26may1108}
    \h t^{1/2}r_\Gamma^{(d-2)/2}\h t^{-1} \asymp \h t^{-(d-1)/2} \log^{(d-2)/2} \h t.
  \end{equation}
  Furthermore, as in \eqref{26may1106}, for any $\theta_1\in\partial\Gamma$ we have
  \begin{equation}
    \label{26may1110}
    2\h t + g(\h t) - \h \gamma - \theta_1\cdot \h x=b_{\h t}- \theta_1 \cdot \h x=b_{\h t}(1-\h\theta\cdot\theta_1)
  \end{equation}
  and hence
  \begin{equation}
    \label{26may1112}
    \e^{-(2 \h t - \h \gamma - \theta_1\cdot \h x)} = \e^{g(\h t)-b_{\h t}(1-\h\theta\cdot\theta_1)} = \hat t^{(d-1)/2} \e^{-b_{\h t}(1-\h\theta\cdot\theta_1)}.
  \end{equation}
  Drawing on \eqref{26may1108}, \eqref{26may1110} and \eqref{26may1112}, the Laplace method applied to \eqref{26may1102} now gives
  \begin{equation}
    \label{eq:center}
    V(\h t, \h x; t) \lesssim (\log \h t)^{\frac{d-2}{2}} v\big(\h{t}, b_{\h t}(1 - \theta_1 \cdot \h \theta) + 1\big) \e^{-b_{\h t}(1 - \theta_1 \cdot \h \theta)},
  \end{equation}
  for any $\theta_1 \in \partial \Gamma$.
  Also, \eqref{eq:Gamma-rad} yields
  \begin{equation*}
    1 - \theta_1 \cdot \h\theta = 1 - \cos \rho_\Gamma = \frac{A^2}{2} \h t^{-1} \log \h t + \m{O}(\h t^{-2} \log^2 \h t).
  \end{equation*}
  Since $b_{\h t} = 2\h t + \m{O}(\log \h t)$, it follows that
  \begin{equation*}
    b_{\h t}(1 - \theta_1 \cdot \h \theta) = A^2 \log \h t + \m{O}(\h t^{-1} \log^2 \h t).
  \end{equation*}
  By Lemma~\ref{lem:drift-heat}, we see that \eqref{eq:center} yields
  \begin{equation*}
    V(\h t, \h x; t) \lesssim A^2 (\log \h t)^{d/2} \h t^{-A^2}.
    \qedhere
  \end{equation*}
  Then \eqref{26may1116} follows from \eqref{26may1114}.
\end{proof}
Finally, we can prove our main localization result.
\begin{proof}[Proof of Proposition~\ref{prop:local}]
  Fix $K > 0$.
  We wish to show that 
  \begin{equation}
    \label{26may1118}
    \P[\h t^K \h Y_t(\Gamma^\cc) \geq 1 \mid \m{G}_{U,T}] \le T^{-K}.
  \end{equation}
  Observe that
  \begin{align*}
    \E \exp\bigl[-\h t^K \h Y_t(\Gamma^\cc) \mid \m{G}_{U,T}\big] &\le \P\big[\h t^K \h Y_t(\Gamma^\cc) \le 1 \mid \m{G}_{U,T}\big] +
                                                                    \e^{-1} \P\big[\h t^K \h Y_t(\Gamma^\cc) \ge 1 \mid \m{G}_{U,T}\big]\\
                                                                  &\leq 1-(1 - \e^{-1})\P\big[\h t^K \h Y_t(\Gamma^\cc) \ge 1 \mid \m{G}_{U,T} \big].
  \end{align*} 
  Rearranging, \eqref{eq:McKean-again} yields
  \begin{equation*}
    \P[\h t^K \h Y_t(\Gamma^\cc) \geq 1 \mid \m{G}_{U,T}] \lesssim 1 - \E \exp\bigl[-\h t^K \h Y_t(\Gamma^\cc) \mid \m{G}_{U,T}\big] = u(\h t, b_{\h t}, \h\theta; t, \h t^K).
  \end{equation*}
  Applying Lemma~\ref{lem:lacuna-small} with $\lambda=\h t^K$, we obtain
  \begin{equation*}
    \P[\h t^K \h Y_t(\Gamma^\cc) \geq 1 \mid \m{G}_{U,T}] \leq C(d) A^2 (\log \h t)^{d/2} \h t^{K - A^2},
  \end{equation*}
  and this bound is uniform in $t \geq \h t$.
  Now let $A = 2\sqrt{K}$.
  Because $\h t \geq T$, \eqref{26may1118} holds for $T \geq C(d,K)$.
\end{proof}

\appendix

\section{Martingale convergence}
\label{sec:convergence}
In this appendix we connect results of~\cite{BerKimLubMalZei24} to prove Theorem~\ref{thm:convergence} and \eqref{eq:positive}.
Recall the windowed population $\m{P}_t^{\text{win}}$ from \eqref{eq:windowbis}.
Adjusting for our diffusivity $\sqrt{2} \op{Id}$, Theorem~1.4 of~\cite{BerKimLubMalZei24} states that
\begin{equation}
  \label{eq:win-conv}
  Y_t^{\text{win}} \coloneqq \pi^{\frac{d-1}{4}} \sum_{p \in \m{P}_t^{\text{win}}} [2t - R_t(p)] \e^{-[2t + \tfrac{d-1}{2} \log t - R_t(p)]} \delta_{\Theta_t(p)} \to Z_\infty
\end{equation}
weakly in probability as $t \to \infty$.
Moreover, \cite[Lemma~3.2]{BerKimLubMalZei24} states that particles outside the window can be safely neglected, as their contributions would vanish in the limit.
Strengthening the statement of Lemma~3.2 of~\cite{BerKimLubMalZei24} by a log-factor, the proof implies that
\begin{equation}
  \label{eq:negligible-log}
  (\log t)\!\! \sum_{p \in \m{P}_t \setminus \m{P}_t^{\text{win}}} (\abs{2t - R_t(p)} + 1) \e^{-[2t + \tfrac{d-1}{2} \log t - R_t(p)]} \overset{\P}{\to} 0 \text{ as } t \to \infty.
\end{equation}
Now define the approximate \emph{additive} martingale
\begin{equation}
  \label{eq:additive}
  A_t \coloneqq \pi^{\frac{d-1}{4}} \sum_{p \in \m{P}_t} \e^{-[2t + \tfrac{d-1}{2} \log t - R_t(p)]} \overset{\P}{\to} 0 \text{ as } t \to \infty.
\end{equation}
Noting that $2t - R_t(p) \geq t^{1/6}$ for $p \in \m{P}_t^{\text{win}}$, we can combine \eqref{eq:win-conv} and \eqref{eq:negligible-log} to conclude that $(\log t)A_t \to 0$ in probability as $t \to \infty$.

Turning to \eqref{eq:shaved}, the vanishing of $(\log t) A_t$ and \eqref{eq:negligible-log} yield
\begin{equation*}
  \norm{Y_t - Y_t^{\text{win}}}_{\mathrm{TV}} \lesssim (\log t) A_t + \sum_{p \in  \m{P}_t \setminus \m{P}_t^{\text{win}}} \abs{2t - R_t(p)} \e^{-[2t + \tfrac{d-1}{2} \log t - R_t(p)]} \to 0
\end{equation*}
in probability as $t \to \infty$.
Thus Theorem~\ref{thm:convergence} follows from \eqref{eq:win-conv}.
Similarly,
\begin{equation*}
  \big[Y_t - \h t^{-1/2} \e^{\beta_{\h t}} \h Y_t\big]_-(S^{d-1}) \lesssim (\log \h t) A_t \overset{\P}{\to} 0 \text{ as } t \to \infty,
\end{equation*}
which justifies \eqref{eq:positive}.
\qed

\printbibliography
\end{document}